\documentclass[11pt,fleqn]{amsart}

\usepackage{epsfig}
\usepackage{amsmath}
\usepackage{amssymb}
\usepackage{amscd}
\usepackage{latexsym}
\usepackage{tabularx}
\usepackage{a4wide}
\usepackage[usenames]{color}
\usepackage{enumerate}
\usepackage{subfigure}
\usepackage{url}

\usepackage{url}

\usepackage[matrix,arrow,tips,curve]{xy}
\usepackage{pb-diagram,pb-xy}
\usepackage{verbatim}
\usepackage{mathrsfs}
\usepackage{color}

\newcommand{\w}{\mathbf w}

\small\normalsize

\newtheorem{thm}{Theorem}[section]

\newtheorem{lemma}[thm]{Lemma}

\newtheorem{prop}[thm]{Proposition}

\newtheorem{claim}[thm]{Claim}
\newtheorem{cor}[thm]{Corollary}

\newtheorem{remark}[thm]{Remark}

\newtheorem{defi}[thm]{Definition}

\newcommand{\Z}{\mathbb{Z}}

\newcommand{\R}{\mathbb{R}}

\newcommand{\N}{\mathbb{N}}

\newcommand{\G}{G}

\newcommand\fX{\mathfrak X}

\renewcommand{\G}{\mathscr G}
\newcommand{\E}{\mathscr E}

\renewcommand{\H}{\mathscr H}

\newcommand{\rat}{\mathrm{rat}}
\newcommand{\Ch}{\mathrm{Chow}_{\mathrm{\tiny GS}}}
\renewcommand{\v}{\mathbf{v}}
\renewcommand{\u}{\mathbf u}
\newcommand{\e}{\mathbf e}

\newcommand{\Rat}{\mathscr I_{\rat}}

\title{The combinatorial Chow ring of products of graphs}
\author{Omid Amini}
\address{CNRS - D\'epartement de math\'ematiques et applications, \'Ecole Normale Sup\'erieure, Paris}
\email{oamini@math.ens.fr}
\begin{document}
\maketitle
\begin{abstract} 
We prove results describing the structure of a Chow ring associated to a product of graphs, which
arises from the Gross-Schoen desingularization of a product
of regular proper semi-stable curves over discrete valuation rings.  
By the works of Johannes Kolb and Shou-Wu Zhang,  this ring controls the behavior of 
the non-Archimedean height pairing on products of smooth proper curves over non-Archimedean fields.  We provide a complete description of the degree map, and 
prove vanishing results affirming a conjecture
of Kolb, which, combined with his work, leads to an analytic formula for the arithmetic intersection number of adelic
metrized line bundles on products of curves over complete discretely valued fields.  
\end{abstract}

\section{Introduction}
 Let $R$ be a complete discrete valuation ring with an algebraically closed residue field $k$ and fraction field $K$, and let $X$ be a smooth proper curve over $K$. 
 By semi-stable reduction theorem, replacing $K$ with a finite extension if necessary, we can find  a regular proper strict semi-stable model $\fX$ of $X$ over the valuation ring. Denote by 
 $\fX_s$ the special fiber of $\fX$, and let $G=(V, E)$ be the dual graph of $\fX_s$. For each vertex $v\in V$, denote by $X_v$ the corresponding 
irreducible component of $\frak X_s$.  The intersection products $X_v.X_u$ are described by the graph $G$ as follows:
\begin{align}\label{pairing:graph} 
\forall \,\,u,v\in V, \qquad X_u.\,X_v = \begin{cases}
             \textrm{number of edges }\{u,v\}  \textrm{ in } G\, & \textrm{ if } u\neq v,\\
             -\mathrm{val}_G(v) & \textrm{ if } u=v,
            \end{cases} 
\end{align}
where $\mathrm{val}_G(v)$ is the valence of $v$ in $G$. The intersection products satisfy the following two sets of relations:
\begin{itemize}
 \item[($\mathscr A1$)] For all $u,v\in V$, $X_v X_u = 0$ if $\{u,v\} \notin E$;
 \item[($\mathscr A2$)] For all $u\in V$, $X_u (\sum_{v\in V} X_v) =0$.
 \end{itemize}

 Consider the polynomial ring $Z(G) = \Z[C_v\,|\,{v\in V}]$ on variables $C_v$, and define the ideal $\mathscr I_{\rat} \subseteq Z(G)$ of elements \emph{rationally equivalent to zero} as the ideal generated by the polynomials 
 $C_uC_v$, for $\{u,v\} \notin E$, and $C_u (\sum_{v\in V} C_v)$ for $u\in V$. Define the Chow ring $\Ch(G)$ of the graph $G$ by $\Ch(G):=Z(G) / \mathscr I_{\rat}$, and note that we have a morphism of graded rings
 $\Ch(G) \rightarrow \mathrm{Chow}^c_{\fX_s}(\fX)$, where $\mathrm{Chow}^c_{\fX_s}(\fX)$ denotes the subring of the Chow ring with support $\mathrm{Chow}_{\fX_s}(\fX)$  of $\fX$ with support in $\fX_s$ generated by the 
 irreducible components of the special fiber. 
  
 For a graph consisting of a single edge $e=\{u,v\}$ on two vertices, we have $\Ch(e) = \mathbb Z[C_u,C_v]/(C_u^2+C_{u}C_v, C_v^2 + C_{u}C_v) \simeq \mathbb Z \oplus \mathbb ZC_u\oplus \mathbb Z C_v \oplus \mathbb Z C_u C_v$. For a general graph $G$, the structure of $\Ch(G)$ is completely described by an exact sequence of rings of the form
 \begin{equation}\label{eq:es1}
 0 \rightarrow \Ch(G) \rightarrow \prod_{e\in E} \Ch(e) \rightarrow \prod_{\substack{\{e_1,e_2 \} \in L(G)}} \mathbb Z,
 \end{equation}
 where $L(G)$ denotes the line graph of $G$ (see Definition~\ref{defi:linegraph}). 
  
 We have a (local) degree map $\deg: \Ch(G) \rightarrow \mathbb Z$ which is defined as follows. 
 For any edge $e = \{u,v\}\in E$, define $\deg_e: \Ch(e) \rightarrow \mathbb Z$ which sends an element $x$ of $\Ch(e)$ to the coefficient of $C_uC_v$ in $x$. The degree map $\deg$ is then the composition of the embedding 
 $\Ch(G) \hookrightarrow \prod_{e\in E} \Ch(e)$ with the map 
 $\sum_{e\in E} \deg_e$.  By definition, the degree map coincides with the intersection pairing~\eqref{pairing:graph}, i.e., for all $u, v\in V$, we have $\deg(C_uC_v) = X_u X_v$.
 
\medskip

Our aim in this paper is to provide a generalization of the above picture for the products of proper smooth curves, that we now describe.

\medskip

Let $X_1, \dots, X_d$ be proper smooth curves over $K$, and, replacing $K$ with a finite extension if necessary, consider a regular strict semi-stable model $\frak X_i$ of $X_i$ over the valuation ring for each $i$.
Starting from the product $\fX_1\times_{\mathrm{Spec}(R)}\dots \times_{\mathrm{Spec}(R)} \fX_d$, the
Gross-Schoen desingularization procedure~\cite{GS} provides a regular proper semi-stable model  $\fX$ of the product $X=X_{1}\times \dots \times X_{d}$ over the valution ring $R$. 
The desingularization  depends on the choice of a total order on the components of the special fiber of each $\fX_i$.

Denote by $G_1=(V_1,E_1), \dots, G_d = (V_d, E_d)$ the dual graphs of the special fibers of $\fX_1 \dots, \fX_d$, respectively, and suppose that a total order $\leq_i$ on the vertex set $V_i$ is given for each $i$.
The dual complex of the special fiber of  the Gross-Schoen model
$\fX$ of $X$ is a triangulation of the product $\mathscr G=G_1 \times \dots \times G_d$. When $\G$ is given by its natural cubical
structure with cubes corresponding to the elements of the product 
$\mathscr E=E_1\times \dots \times E_d$, the triangulation consists of the union of the standard triangulation of 
these $d$-dimensional cubes,  compatible with the fixed total orders on the vertex set of each graph $G_i$, see Section~\ref{sec:generalities} for the precise definition. 
The  vertex set $\G_0$ of the simplicial set $\G$ 
is the product $\G_0= V_1\times \dots \times V_d$ of the vertex sets, whose elements are in bijection with  irreducible components of
the special fiber $\fX_s$ of $\fX$: for an element $\v\in \G_0$, we denote by $X_\v$ the corresponding 
irreducible component of $\fX_s$.

Consider the Chow ring with support  $\mathrm{Chow}_{\fX_s}(\fX)$ (see~\cite{Fulton} for the definition of Chow rings with support). The intersection products between $X_\v$, for $\v\in \mathscr V$, in  $\mathrm{Chow}_{\fX_s}(\fX)$ verify three types of 
equations (see $(\mathscr R1), (\mathscr R2), (\mathscr R3)$ below), given by Kolb in~\cite{Kolb1, Kolb2}, two of which are higher dimensional analogous of the relations $(\mathscr A1)$ and $(\mathscr A2)$. 
This leads to the definition of a Chow ring for the product of graphs, that we describe next.
 
 \subsection{Definition of the Chow ring $\Ch(\G)$.} Denote by $Z(\G)$ the polynomial ring with coefficients in $\Z$ generated by the vertices of $\G$, namely, 
 \[Z(\G) := \Z[C_{\v} \,|\, \v\in \mathscr G_0],\]
 where the variables $C_{\v}$ are associated to the vertices (0-simplices) of $\G$.
 We view $Z(\G)$ as a graded ring where each variable $C_\v$ is of degree one.

 The  graded ideal $\mathscr I_{\rat}$ of all the elements of $Z(\G)$ which are \emph{rationally equivalent to zero}  is defined as the ideal generated by the following three types  of generators:
 
 \begin{itemize}
  \item[($\mathscr R1$)] $C_{\v_1} C_{\v_2} \dots C_{\v_k}$ \qquad for $k\in \N$  and elements $\v_j\in \G_0$ such that $\v_1,\dots, \v_k$ do not form a simplex in $\G$\,; 
  \item[($\mathscr R2$)]  $ C_{\u} \Bigl(\,\sum_{\v \in \G_0} C_{\v}\,\Bigr)$\qquad for any vertex  $\u \in \G_0$; and
  \item[($\mathscr R3$)]   $C_{\u} C_{\mathbf{w}}\Bigl(\,\sum_{\v \in \G_0: v_i = u_i}C_{\v}\,\Bigr)\, $ \quad for any pair of vertices $\u, \mathbf{w} \in \G_0$ and any index $1\leq i \leq d$ with $u_i \neq w_i$.
 \end{itemize}
 
 \begin{defi}\rm
  The combinatorial Chow ring of $\G$ is the graded ring $\Ch(\G) := Z(\G)/\Rat$. 
 \end{defi}

 The ring $\Ch(\G)$ is the universal graded commutative ring with generators indexed by vertices of $\G$ and  verifying relations ($\mathscr R1$), ($\mathscr R2$), ($\mathscr R3$) above; in particular 
  we get a well-defined map
  \begin{equation}\label{eq:univ}
  \alpha_\fX: \Ch(\G) \rightarrow \mathrm{Chow}_{\mathfrak X_s}(\mathfrak X),
  \end{equation}
  for $\mathfrak X$ the Gross-Schoen desingularization of the products of curves $\fX_1 \times \dots \times \fX_d$ (where for each $i$, the dual graph of the special fiber of $\fX_i$ is isomorphic to $G_i$).

  \subsection{Statement of the main results} The main contributions of this paper are the structure Theorem~\ref{thm:st} which describes the additive structure of the graded pieces of the Chow ring, 
  the localization Theorem~\ref{thm1}, which is a generalization of the exact sequence~\eqref{eq:es1} to products of graphs, the calculation of the degree map, which is a generalization 
  of~\eqref{pairing:graph} to higher dimension, and a vanishing theorem confirming a conjecture of Kolb. We now  discuss these results.

  \subsubsection{Localization} We prove the localization theorem~\ref{thm1}, a generalization of the exact sequence~\eqref{eq:es1} to products of graphs, which shows that the calculations in the ring 
  $\Ch(\G)$ can be reduced to calculation in the Chow ring of the hypercubes of dimension $d$, namely, the products of $d$ copies of the complete graph $K_2$ on two vertices.

 \medskip

 Recall first that a homomorphism of graphs $f: H \rightarrow G$ is a map $f: V(H) \rightarrow V(G)$ such that for any edge $\{u,v\}\in E(H)$, either $f(u) =f(v)$ or $\{f(u), f(v)\}\in E(G)$. 
 
Let $H_1,\dots, H_d$ be $d$ simple connected graphs. Define $\mathscr H= H_1\times \dots \times H_d$ with the induced simplicial structure. Suppose that for each $i=1, \dots, d$, a homomorphism of graphs 
 $f_i : H_i \rightarrow G_i$ is given (such that $f_i$ respects also the two fixed orderings on the vertex set of $H_i$ and $G_i$).
  The product of $f_i$ leads to 
 a morphism of simplicial sets 
 $f: \mathscr H \rightarrow \G$, and induces a morphism of graded rings $f^* :  Z(\G) \rightarrow Z(\H)$, which is defined on the level of generators by sending $C_\v$ for $\v \in \G_0$ to
 \[f^*(C_\v) = \sum_{\substack{\u\in \mathscr H_0\\ f(\u) =\v}} C_\u.\]
 It is not hard to see that the map $f^*$ sends $\Rat(\G)$ to $\Rat(\H)$ and induces a well-defined map of Chow rings $f^*: \Ch(\G) \rightarrow \Ch(\H)$, c.f., Proposition~\ref{prop:func}.
 
 Let now $G_1=(V_1, E_1), \dots, (G_d, E_d)$ be a collection of $d$ simple connected graphs, and $\G = \prod_i G_i$, as above.
 For $\e\in \E = E_1 \times \dots \times E_d$, let $\square_\e =e_1\times \dots \times e_d$.  
  Regarding each edge $e_i$ as a subgraph of $G$ isomorphic to $K_2$, 
  and applying the functoriality to the inclusions of the subgraph $e_i \hookrightarrow G_i$, we get 
  a map $\iota_{\e}^*: \Ch(\G) \rightarrow \Ch(\square_\e) \simeq \Ch(\square^d)$ associated to the 
  inclusion map of simplicial sets
 $$\iota_{\e}:\square_{\e} \hookrightarrow \G.$$ 
 By definition, the map $\iota^*_\e$ is identity on the generators associated to the vertices of 
 $\square_\e$, and is zero otherwise.

 Write $e_i = u_iv_i$ for vertices $u_i <v_i$ of $G_i$ with respect to the total order of $G_i$, and define the two corresponding facets $\square_{\e,{0}_i }$ and $\square_{\e,{1}_i}$ of $\square_\e$ associated to $u_i$ and $v_i$, respectively, by 
 \begin{align*}
 \square_{\e,0_i}&:=e_1\times \dots \times e_{i-1}\times \{u_i\} \times e_{i+1} \times \dots \times e_d, \textrm{ and }\\
 \quad \square_{\e,1_i}&:=e_1\times\dots \times e_{i-1}\times \{v_i\} \times e_{i+1} \times \dots \times e_d.
 \end{align*}
 Let $\iota_{\e,0_i}:\square_{\e,0_i} \hookrightarrow \square_{\e}$, and similarly,  $\iota_{\e,1_i}:\square_{\e,1_i} \hookrightarrow \square_{\e}$, the inclusion maps. 
By functoriality, we get maps 
$\iota_{\e,0_i}^* : \Ch(\square_\e)  \rightarrow \Ch(\square_{\e,0_i})$ 
and $\iota_{\e,1_i}^* : \Ch(\square_\e)  \rightarrow \Ch(\square_{\e,1_i})$ on the level of Chow rings.  

\medskip
We next recall the definition of the line graph of a give graph.

\begin{defi}\label{defi:linegraph} \rm Let $G = (V,E)$ be a given simple connected graph on vertex set $V$ and edge set $E$. 
The \emph{line graph} of $G$ denoted by 
$L(G)$ is the graph on vertex set $E$ and with edge set  
consisting of all the pairs $\{e,e'\}\subset E$ with $e$ and $e'$ incident edges in $G$. 
\end{defi}

Let $G_1 = (V_1,E_1), \dots, G_d=(V_d,E_d)$ be a collection of $d$ simple connected graphs as before. 
For each $i=1,\dots, d$, define the set 
$\E_i:=E_1\times \dots \times E_{i-1} \times E(L(G_i)) \times E_{i+1} \times \dots \times E_d$. 
Thus, an element $\mathbf x$ of 
$\E_i$ is a collection $e_j \in E_j$, for $j\neq i$, and $\{e_{i,1}, e_{i,2}\}$ with $e_{i,1},e_{i,2} \in E_i$ and 
$e_{i,1} \cap e_{i_2} \neq \emptyset$. The element $\mathbf x \in \E_i$ therefore gives two hypercubes 
$\square_{\e_1}$ and $\square_{\e_2}$ in $\G$, with 
$\e_k=e_1 \times \dots \times e_{i-1}\times e_{i,k} \times \dots \times e_{d}$, for 
$k=1,2$. Note that $\square_{\e_1}$ and $\square_{\e_2}$ share the facet $\square_{\mathbf x}:= e_1 \times \dots \times e_{i-1} 
\times (e_1 \cap e_2) \times e_{i+1} \times \dots \times e_d$. Denote by $\iota_{{\mathbf x},1}$ and $\iota_{{\mathbf x},2}$ the inclusion of 
$\square_{\mathbf x}$ in $\square_{\mathbf e_1}$ and $\square_{\mathbf e_2}$, respectively. 
Denoting by $j_{{\mathbf x}}^*: 
\Ch(\square_{\e_1}) \times \Ch(\square_{\e_2}) \rightarrow \Ch(\square_{\mathbf x})$ the map which sends 
the pair $(\alpha, \beta)$ to $\iota_{{\mathbf x},1}^*(\alpha) -  \iota_{{\mathbf x},2}^*(\beta)$, 
we get a map 
$$j :  \prod_{\e\in \E} \Ch(\square_{\e}) \longrightarrow \prod_{i=1}^d \prod_{\mathbf x \in \E_i} \Ch(
\square_{\mathbf x}).$$
Note that the $\mathbf x$-coordinate of $j$ is the composition of 
$j_{\mathbf x}^*$ with the projection from 
$$\prod_{\e\in \E} \Ch(\square_{\e}) \to \Ch(\square_{\e_1}) \times \Ch(\square_{\e_2}).$$
\begin{remark}\rm 
The map $j_{\mathbf x}^*$, and so $j$, is well-defined only up to the sign consisting in changing the role 
of $e_1$ and $e_2$. In order the get a well-defined map, we can fix a total order on the set $E_i$, 
for $i=1, \dots, d$, and for any edge $\{e_{i,1}, e_{i,2}\}$ of $L(G_i)$ require that $e_{i,1}< e_{i,2}$ with respect to this total order.
We remark that the choice of the sign is irrelevant for what follows.
\end{remark}

 \medskip
 
With these notations, we can state our localization theorem. 
 \begin{thm}\label{thm1} The map of graded rings
$  \prod_{{\bf e} \in \mathcal E}\, \iota_{\bf e}*: \Ch(\G) \rightarrow \bigoplus_{\e \in \E} \Ch(\square_{\mathbf e})$ is injective and identifies  $\Ch(\G)$ with the kernel of the map
\[j:\prod_{\e \in \E} \Ch(\square_{\e}) \longrightarrow \prod_{i=1}^d \prod_{\mathbf x \in \E_i} \Ch(
\square_{\mathbf x}).\]
 \end{thm}

In other words, the ring $\Ch(\G)$ is the inverse limit of the Chow ring of cubes of $\G$ for the diagram of maps induced by the inclusion of cubes.
Endowing the simplicial set $\G$ with the cubical topology with a basis of open sets consisting of the products of subgraphs of $G_i$, the localization theorem ensures that the Chow rings 
 form a sheaf for the coverings of $\G$ with open sets whose union covers all the simplices of $\G$.

 \medskip
 
 The proof of this theorem is given in Section~\ref{sec:loc}.

\subsubsection{Description of the additive structure of $\Ch(\G)$} 
The proof of the localization theorem is indirect, and is based on a structure theorem which provides a description of the Chow groups in terms of 
non-degenerate simplices of the product $\G$ and specific relations, taking into account the cubical structure of $\G$. We now describe this.

\medskip

The total orders $\leq_i$ on the vertex sets $V_i$, $i=1,\dots, d$, induce a partial order $\leq$ on $\G_0$ defined by saying $\u=(u_1, \dots, u_d) \leq \v=(v_1, \dots, v_d)$ if $u_i \leq_i v_i$ for each $i$.

\medskip

  Let $\u<\v$ be two elements of $\G_0$ such that $\{\u,\v\}$ forms a one-simplex. This means that for each $i$, we have $u_i=v_i$ or $\{u_i,v_i\} \in E_i$, c.f. Section~\ref{sec:generalities}.
  Denote by $I(\u,\v)$  the set of all indices $i$ with $u_i<v_i$, and let $\e_{\u,\v}$ be the cube of dimension $|I(\u,\v)|$ formed by all the vertices $\mathbf z=(z_1, \dots, z_d)$ in $\G_0$ with 
  $z_i \in \{u_i,v_i\}$, i.e., $\e_{\u,\v} = \prod_{i=1}^d \{u_i,v_i\}$.
  
  \medskip
  
For any integer $k\in \mathbb N$, denote by $\G_k^{nd}$ the set of all non-degenerate $k$-simplices of $\G$. Each element $\sigma$ of $\G_k^{nd}$ is a sequence $\u_0 < \u_1<\dots <\u_k$ of vertices $\u_i\in \G_0$ 
such that $\{\u_j,\u_{j+1}\}$ is a 1-simplex of $\G$, for any $0\leq j\leq k-1$.

 Let $\sigma$ be a non-degenerate $k$-simplex of $\G$. For two indices $1\leq i,j \leq d $ lying both in $I(\u_{t}, \u_{t+1})$ for some $0\leq t \leq r-1$, define 
 \[\widetilde R_{\sigma,i,j} := \sum_{\substack{\w\in \e_{\u_t, \u_{t+1}}\\ w_i =u_{t,i}}} C_\w - \sum_{\substack{\w \in \e_{\u_t,\u_{t+1}}\\ w_j =u_{t,j}}} C_\w.\]
  For any $k$-simplex $\sigma \in \G_k$, denote by $C_\sigma$ the product (with multiplicity) of $C_\v$ over all vertices of $\sigma$, i.e., $C_\sigma:=\prod_{\v \in \sigma} C_\v.$ 
  Using $(\mathscr R1)$ and $(\mathscr R3)$, one verifies that $C_\sigma \widetilde R_{\sigma, i,j} \in \mathscr I_{\rat}$. 
 
 \medskip

 For any $k\in \mathbb N$, denote by $\mathbb Z\langle \G_{k}^{nd}\rangle$ the $\mathbb Z$-submodule of $Z(\G)$  generated by all elements $C_{\sigma}$ for $\sigma \in \G_k^{nd}$.
 By definition of the simplicial structure, one sees that for an element $\sigma \in \G_{k}^{nd}$ consisting of vertices $\u_0<\dots <\u_k$, and for $i,j \in I(\u_t, \u_{t+1})$ as above, the product 
 $C_\sigma \widetilde R_{\sigma, i,j}$ lies in
 $\mathbb Z\langle \G_{k}^{nd}\rangle$. Denote by $\mathscr I^{nd}_{k}$ the $\mathbb Z$ submodule of $\mathbb Z\langle \G_{k}^{nd} \rangle $ generated by all the elements 
 $C_\sigma \widetilde R_{\sigma, i,j}$, for $\sigma \in \G_k^{nd}$ and $i,j\in I(\u_t, \u_{t+1})$ as above.  We have the following theorem.
 \begin{thm}\label{thm:st} For any non-negative integer $k$, we have 
\[\Ch^{k+1}(\G) \simeq \mathbb Z\langle \G_{k}^{\textrm{nd}}\rangle / \mathscr I^{nd}_{k}. \]
\end{thm}

 The existence of a surjection from $\mathbb Z\langle \G_{k}^{\textrm{nd}}\rangle / \mathscr I^{nd}_{k}$ to $\Ch^{k+1}(\G)$ is a consequence of the moving lemma, c.f. Theorem~\ref{thm:nd}. 
 The proof of the injectivity of this map, on the other hand, turns out to be quite tricky and technical.   This is given in Section~\ref{sec:str}.
 \subsubsection{Combinatorics of the degree map}
 Let $\square^d = \{0,1\}^d$, the $d$-dimensional hypercube, be $d$-fold product of the complete graph $K_2$ on two vertices $0<1$ with its standard simplicial structure.  
 It follows from the structure theorem that the Chow ring $\Ch(\square^d)$ is of rank one in graded degree $d+1$, i.e., $\Ch^{d+1}(\square^d) \simeq \mathbb Z$, generated by $C_\sigma$ for any non-degenerate $d$-simplex  
 $\sigma$ of $\square^{d}$.  This leads to a well-defined degree  map 
\[\deg: \Ch^{d+1}(\square^d) \rightarrow \Z.\]
 Combining this with the localization theorem, and the vanishing of $\Ch^{i}(\square^d)$ in degree $i\geq d+2$, which follows e.g. from the structure theorem, or the moving lemma, we infer that for any collection of 
 simple connected graphs $G_1 = (V_1,E_1), \dots, G_d=(V_d, E_d)$, we have $\Ch^{d+1}(\mathscr G) \simeq \mathbb Z^{|\mathscr E|}$. Therefore we get a degree map $\deg: \Ch^{d+1}(\G) \simeq \mathbb Z^{|\mathscr E|} \rightarrow \mathbb Z$, 
 by aditionning the coordinates in $\mathbb Z^{|\mathscr E|}$.

\medskip

Our next result gives a combinatorial formula for the value of the degree map.  
Combined with the map $\alpha_\fX$ in~\eqref{eq:univ}, this results in a concrete effective description of the local degrees in the Chow ring 
$\mathrm{Chow}^c_{\fX_s}(\fX)$ for the Gross-Schoen desingularization $\fX$ of a product of semi-stable $R$-curves $X_1, \dots, X_d$, generalizing~\eqref{pairing:graph} to higher dimension.

\medskip

Since $\Ch^{d+1}(\G)$ is generated by monomials, we can restrict to the case of a  monomial. In addition, by localization theorem, and the definition of the degree map, 
it will be enough to treat the case of the hypercube $\square^d$.

\medskip

 By definition of the simplicial structure, each (possibly degenerate) $d$-simplex $\sigma$ of $\square^d$ is of the form 
$\v_1^{n_1}\v_{2}^{n_2}\dots \v_{k}^{n_k}$ with $\v_1<\v_2<\dots<\v_k$, and $n_i\geq 1$ with $\sum_i n_i =d+1$.  We have $0\leq |\v_1|<\dots<|\v_k| \leq d$, where for any $\v\in \square^d$,  the \emph{length} $|\v|$ of $\v$ denotes  the number of coordinates of $\v$ equal to one.
Consider the set $[d] := \{0,1,\dots, d\}$. Let us say a point $|\v_i|$ is a \emph{neighbor} of a point $x \in [d] \setminus \bigl\{\,|\v_1|, \dots, |\v_k|\,\bigr\}$ if the interval formed by $x$ and $|\v_i|$ does not contain any other point of $\bigl\{\,|\v_1|,\dots, |\v_k|\,\bigr\}$ beside $|\v_i|$. In this way each point $x \in [d] \setminus \bigl\{\,|\v_1|, \dots, |\v_k|\,\bigr\}$ has either one or two neighbors among the points $|\v_1|,\dots, |\v_k|$.

Assume now that $n_i$ chips are placed on the  point $|\v_i|$ in $[d]$. The total number of chips is thus $\sum_i n_i =d+1$. We assume further that the chips are labelled. Each point $|\v_i| \in [d]$ chooses, once for all, one of its $n_i$ chips that she wants to keep, and decide to give all the extra remaining chips to $n_i-1$ of its neighbors in $[d] \setminus \bigl\{\,|\v_1|,\dots,|\v_k|\,\bigr\}$  in such a way that at the end, each point of $[d]$ has precisely one chip.
  In how many ways this can be done? 
 The following theorem states that, up to a $\pm$ sign, the degree of $\v_1^{n_1}\dots \v_k^{n_k}$ in $\Ch(\square^d)$ is given by this number. 
  
\begin{thm}\label{thm:main1-intro} Notations as above, let $C_\sigma = C_{\v_1}^{n_1}\dots C_{\v_k}^{n_k}$. 
One of the two following cases can happen.
\begin{enumerate}
 \item If there exists an $1\leq i <k$ such that $n_1+\dots+n_i > |\v_{i+1}|$, then   $C_{\sigma} =0$. Similarly, if there exists an $k\geq i \geq 2$ such that $n_{i}+ \dots + n_k > d-|\v_{i-1}|$, then $C_\sigma =0$.
 \item Otherwise, there exists  a sequence of integers $y_0, x_1,y_1, x_2, y_2, \dots, x_{k-1}, y_{k-1}, x_k$ verifying the following properties
 \begin{itemize}
   \item $|\v_1| = y_0$.
   \item For all $i=2,\dots, k$, $|\v_i| = |\v_{i-1}|+x_i+y_i+1$.
  \item $n_i = y_{i-1}+x_{i}+1 $ for all $i=1,\dots, k$,
   \end{itemize}
   and, in this case, we have
   \[\deg(C_\sigma) = (-1)^{d+1-k}\binom{y_0+x_1}{y_0} \binom{x_1+y_1}{x_1}\binom{y_1+x_2}{y_1} \dots \binom{x_{k-1}+y_{k-1}}{y_{k-1}}\binom{y_{k-1}+x_k}{x_k}.\]
\end{enumerate}
\end{thm}
  Note that the product $\binom{y_0+x_1}{y_0} \binom{x_1+y_1}{x_1}\binom{y_1+x_2}{y_1} \dots \binom{x_{k-1}+y_{k-1}}{y_{k-1}}\binom{y_{k-1}+x_k}{x_k}$ in (2) is precisely the number of ways the extra (labelled) chips can be placed on the points of $[d] \setminus \{\ell_1,\dots, \ell_k\}$ so that each point receives precisely one chip from one of its neighbors. Otherwise, in case (1), it is not possible to place the chips ensuring to have one chip at each point of $[d]$.
  
\subsubsection{Fourier transform and a vanishing theorem}  
 Identifying the points of $\square^d$ with the elements of the vector space $\mathbb F_2^d$, it is possible to give a dual description of the Chow ring of the hypercube using the Fourier duality. 
 So let $\langle\,,\rangle$ be the scalar product on $\mathbb F_2^d$ defined by 
\[\forall\,\u,\v\in \mathbb F_2^d,\qquad \langle\v,\u\rangle  := \sum_{i=1}^d v_i.u_i \in \mathbb F_2.\]
For $\w \in \mathbb F_2^d$, define $F_\w$ by
\[F_\w := \sum_{\v\in\square^d} (-1)^{\langle \v,\w\rangle}C_\v.\]
By Fourier duality, we have for any $\v\in \mathbb F_2^d$, 
\[C_\v = \frac 1{2^d} \sum_{\v\in\square^d} (-1)^{\langle \v,\w\rangle }F_\w.\] It follows that the set $\{F_\w\}_{\w\in \square^d}$ forms another system of generators for the  
Chow ring $\Ch(\square^d)[\frac 12]$ localized at 2, that we call the Fourier dual of the set $\{C_\v\}_{\v\in \square^d}$. 

Denote by $\bf 1$, and $\bf 0$, the points of $\square^d$ whose coordinates are all equal to one, and zero, respectively.  Let $\e_i$ be the point of $\mathbb F_2^d$ with $i$-th coordinate equal to 1, and all the other coordinates equal to $0$. 
Kolb proved in~\cite{Kolb1}  that the following set of relations are verified by $\{F_\w\}$ in $\Ch(\square^d)$:

\begin{itemize}
 \item[($\mathscr R^* 1$)] For any $\bf w \in \mathbb F^d$, we have $F_{\bf 0} F_{\bf w}=0$;
 \item[($\mathscr R^*2$)] For any $i\in [d]$, and any $\bf w, \bf z \in \mathbb F^d$, we have $F_{\e_i}(F_{\bf w}-F_{\w+\e_i})(F_{\bf z}+ F_{\bf z+\e_i})=0$;
 \item[($\mathscr R^*3$)] For any pair of indices $i,j\in[d]$, and any $\bf w, \bf z$, we have $(F_{{\bf w}+\e_i+\e_j}-F_{\bf w})(F_{{\bf z}+\e_i+\e_j}-F_{\bf z}) = (F_{{\bf w}+\e_i}-F_{{\bf w}+\e_j})(F_{{\bf z}+\e_i}-F_{{\bf z}+\e_j})$.
\end{itemize}

We have the following more precise statement. Let $\widetilde \Rat(\square^d)$ be the ideal of $\Z[F_{\bf w}]_{{\bf w}\in \square^d}$ generated by the relations given by ($\mathscr R^*1$), ($\mathscr R^*2$), and ($\mathscr R^*3$) above, and define 
$$\widetilde \Ch(\square^d) := \Z[F_{\bf w}]_{\w\in \square^d}\,/\,\widetilde \Rat(\square^d).$$
\begin{thm}\label{thm:iso}
 The set of relations $(\mathscr R^* 1)$, $(\mathscr R^* 2)$, and $(\mathscr R^* 3)$ generate the ideal $\Rat(\square^d)$ in $\Z[\frac 12][C_{\bf v}]_{\v \in \square^d} = \Z[\frac 12][F_{\bf w}]_{\bf w\in \square^d}$. In particular,  we have $\widetilde \Ch(\square^d)[\frac 12]  =\Ch(\square^d)[\frac 12]$. 
\end{thm}
We now describe 
a criterion guaranteeing  the vanishing of  a monomial of the form $F_{\w_0} \dots F_{\w_d}$, for elements $\w_0,\dots, \w_d \in \mathbb F_2^d$. \medskip

  Let $\mathcal P = \{P_1,\dots,P_k\}$ be  a partition of $\{1,\dots, d\}$ into $k$ disjoint non-empty sets. For each $\w_i$, denote by 
$\alpha(\w_i,\mathcal P)$ the number of indices $1\leq i\leq k$ such that there exists $j\in P_i$ with $\w_j=1$.
 \begin{thm}\label{thm:vanishing-intro}
 If $\sum_{i=0}^d \alpha(\w_i, \mathcal P) <d+k$, then we have $F_{\w_0}\dots F_{\w_d}=0$ \textrm{in the Chow ring}. 
 \end{thm}
 This property was conjectured by Kolb and is required in~\cite{Kolb2} in order to get the analytic description of the local degree map, that we briefly describe in the next section.

  \subsection{Analytic description of the local intersection numbers}
  
  Let $X$ be a smooth proper curve over a complete discretely valued field $K$ with an algebraically closed residue field. 
  The Berkovich analytification $X^{\mathrm{an}}$ of $X$ is a compact path-wise connected Hausdorff topological space which deformation retracts to a compact metric graph $\Gamma$~\cite{Ber90, BPR, Ducros}. 
  If $X$ admits a regular semi-stable model $\fX$ over the valuation ring of $K$, the metric graph $\Gamma$ has a model $(G, \ell)$ given by the dual graph $G=(V,E)$ of $\fX$ and 
  the edge length $\ell: E(G) \rightarrow \mathbb R$ given by $\ell(e)=1$ for all edges $e\in E$. (So $\Gamma$ is the metric realization of $(G,\ell)$ in the sense that each edge in $G$ is 
  replaced with an interval of length one, see e.g.~\cite{Ami, BR}.) 
  
  Any Cartier divisor $D$ on $\fX$ with support in the special fiber $\fX$ gives a map $f: V \rightarrow \mathbb Z$, that we can extend to $\Gamma$ by linear interpolation on interior points of the 
  intervals in $\Gamma$ corresponding to the edges of $G$. For two Cartier divisors $D_1, D_2 \in \mathrm{Chow}^1_{\fX_s}(\fX)$ with functions $f_1, f_2 : \Gamma \rightarrow \mathbb R$, the degree map given by the pairing~\eqref{pairing:graph} gives a number $\deg(D_1 D_2)$, which can be described analytically as 
  \begin{equation} \label{analyticdeg1}
  \deg(D_1D_2) = \langle f_1, f_2\rangle_{\mathrm{Dir}} = -\int_{\Gamma} f_1'f_2'.
  \end{equation}
  Here $\langle.\,,.\rangle_{\mathrm{Dir}}$ denotes the Dirichlet pairing on piecewise smooth functions on $\Gamma$~\cite{BPR, Zhang-ad}.

 By an approximation argument involving semi-stable models of curves $X_{K'}$ for finite extensions $K'/K$, and viewing Cartier divisors with support in the special fibers of semi-stable models of $X_{K'}$ 
  as piecewise linear functions on $\Gamma$, one can continuously extend the (degree) pairing between divisors to the full class of piecewise smooth functions on $\Gamma$ such that the equation above remains valid for 
  this more general class of functions~\cite{Zhang-ad}.

  \medskip
  
  Motivated by applications in arithmetic geometry, Zhang derived in~\cite{Zhang} a generalization of the analytic formula~\eqref{analyticdeg1} for the degree pairing in the case of a 2-fold product of a smooth proper 
  curve $X$ over $K$.  Kolb~\cite{Kolb2} later generalized this to $d$-fold 
  products of $X$ assuming the validity of the vanishing Theorem~\ref{thm:vanishing-intro}. We state his result in the more general setting of a product of smooth proper curves $X_1, \dots, X_d$. 
  
  Let $X_1, \dots, X_d$ be smooth proper curves over $K$, that we suppose (up to passing to a finite extension of $K$), to have 
  regular strict semi-stable models $\fX_1, \dots, \fX_d$ over the valuation ring $R$. Denote by $G_1, \dots, G_d$ the dual graph of the special fibers of $\fX_1, \dots, \fX_d$, and let $\G = G_1 \times \dots \times G_d$ 
  with its simplicial structure. 
  The Gross-Schoen desingularization gives a regular proper strict
   semi-stable model $\fX$ of the product $X= X_1 \times \dots \times X_d$ with special fiber $\fX_s$ having a dual complex isomorphic to $\G$.  
   The geometric realization of $\G$ is a locally affine space which embeds in the Berkovich analytification $X^{\mathrm{an}}$ of $X$ (by a general theorem of Berkovich~\cite{Ber99}). 
   Each Cartier divisor $D$ with support in the special fiber $\fX_s$ induces
 a metric on the trivial line bundle corresponding to a piecewise affine
function on the geometric realization of $\G$ (for metric on lines bundles see e.g.~\cite{Zhang-sm}).  
   The intersection pairing given by the degree map induces a multi-linear pairing between piecewise affine
functions (by passing to finite extensions of $K$ if necessary), and we have 
   \[\langle f_{D_0}, \dots, f_{D_d}\rangle = \deg(D_0\dots D_d).\]
This pairing can be viewed as the local contribution to the intersection product of
metrized line bundles in non-Archimedean Arakelov theory. It is useful to extend
this pairing to a larger class of metrized line bundles. By approximation one may take a sequence of piecewise linear functions converging to each of the piecewise smooth functions $f_{i}$. This has been carried out in great detail in~\cite{Kolb2}. 
The well-definedness of the extension as well as the analytic generalization of the Formula~\eqref{analyticdeg1} is guaranteed if  the vanishing condition in Theorem~\ref{thm:vanishing-intro} holds. 
To state the theorem, we need to introduce some notations.

For each graph $G_i$, denote by $\Gamma_i$ the metric graph associated to $(G,\ell)$ with length function $\ell\equiv 1$, the constant function.  For each $n\in \mathbb N$, denote by $G_i^{(n)} = (V_i^{(n)}, E_i^{(n)})$ the $n$-th subdivision of $G_i$, where each edge $e$ is subdivided into $n$ edges.  The pair $(G_i^{(n)}, \ell^{(n)})$ with length function $\ell^{(n)} \equiv 1/n$ is a model of the same metric graph $\Gamma_i$.  A total order on the vertex set of $G_i$ naturally extends to a total order on the vertex set of $G_i^{(n)}$ such that the vertices of $G_i^{(n)}$ on each edge  $e$ of $G$ form a monotone sequence.  Denote by $\G^{(n)}$ the simplicial set on the product $G_1^{(n)} \times \dots \times G_d^{(n)}$. This provides a triangulation of the topological space $\mathscr T = \Gamma_1\times \dots \times \Gamma_d$.  The space $\mathscr T$ has natural affine structure induced by the cubes $\square_\e \simeq [0,1]^d$  for each $\e \in \mathscr E^{(n)} = E^{(n)}_1 \times \dots \times E^{(n)}_d$. 
Define the space $\mathcal C^{\infty}_{\Delta}(\
\mathscr T)$ as the space of functions $f: \mathscr T \rightarrow \mathbb R$ which are smooth on simplices of $\G$~\cite{Kolb2}. This means, for any cube $\square_\e \simeq [0,1]^d$, the restriction of $f$ to each triangle $\Delta$ of $[0,1]^d$ can be extended to a smooth function in a neighborhood of $\Delta$.  	

For each $f\in \mathcal C^{\infty}_{\Delta}(\mathscr T)$, denote by $f^{(n)}$ the piecewise affine function on $\mathscr T$ obtained by interpolating on each simplex $\sigma$ of $\G^{(n)}$, the values of $f$ on the vertices of $\sigma$ to all the interior points of $\sigma$.
 
The graphs $G_i^{(n)}$ are the dual graphs of a semi-stable model $\fX_i^{(n)}$ of $X_{K'}$ for an appropriate finite extension $K'$ of $K$, and the simplicial set $\G^{(n)}$ corresponds 
to the dual complex of the Gross-Schoen desingularization of the product $\fX_1^{(n)} \times \dots \fX_d^{(n)}$. Looking at $\mathbb R$-Cartier divisors with support in the special fiber $\fX_s^{(n)}$ 
as real valued functions defined on the vertices of $\G^{(n)}$, the degree map in the ring $\mathrm{Chow}_{\fX_{s}}^{\fX^{(n)}}$ leads to a pairing $\langle f_0^{(n)}, \dots, f_d^{(n)}\rangle$ 
for any collection of functions $f_0, \dots, f_d \in \mathcal C^{\infty}_{\Delta}(\mathscr T)$. With these preliminaries, combining our Theorem~\ref{thm:vanishing-intro} with the results in~\cite{Kolb2}, 
we get the following generalization of Equation~\eqref{analyticdeg1}.

\begin{thm}[Kolb~\cite{Kolb2}] For any collection of functions $f_0, \dots, f_d \in \mathcal C^{\infty}_{\Delta}(\mathscr T)$, the limit 
\[\langle f_0, \dots, f_d \rangle := \lim_{n\to \infty} \langle f_0^{(n)}, \dots, f_d^{(n)}\rangle \]
exists, and admits the following analytic development 
\[\langle f_0, \dots, f_d \rangle  =  \sum_{\substack{\textrm{ partitions} \\ {\mathcal P} \textrm{ of }[d]}} \,\,\frac{1}{2^{|\mathcal P|+d}} \sum_{\substack{\w_0,\dots, \w_d \in \mathbb F_2^d\\ \sum \alpha(\w_i,\mathcal P) = d+ |\mathcal P|}} \deg\,\bigl(\prod_{i=0}^d F_{\w_i}\bigr) \int_{\mathrm{Diag}_{\mathcal P}} \prod_{i=0}^d D^{\w_i}_{\alpha(\w_i, \mathcal P)}(f_i).\]
\end{thm}
In the above formula the generalized diagonal $\mathrm{Diag}_{\mathcal P}$ is the union of the generalized diagonals $\mathrm{Diag}^\e_{\mathcal P}$ in the hypercubes $\square_\e \simeq \square^d=[0,1]^d$ consisting of all the points $(x_1, \dots, x_d) \in [0,1]^d$ which verify $x_i=x_j$ for all $i,j\in \{d\}$ belonging to the same element of the partition  $\mathcal P$. The term $D^{\w_i}_{\alpha(\w_i, \mathcal P)}(f_i)$ is a  \emph{partial derivative of $f_i$ of order $\alpha(\w_i, \mathcal P)$  in the direction of $\w_i$ and along the generalized diagonal $\mathrm{Diag}_\mathcal P$}.  For example, for the partition $\mathcal P$ of $[d]$ into singletons, we have $\alpha(\w, \mathcal P) = |\w|$ for any $\w \in \mathbb F_2^d$, and on any cube $\square_\e \simeq [0,1]^d$, we have $D^{\w_i}_{\alpha(\w_i, \mathcal P)}= 
 (\frac{\partial}{\partial x_1})^{w_1}\dots\, (\frac{\partial}{\partial x_d})^{w_d}.$ We omit the formal definition and refer to~\cite{Kolb2} for more details.

\medskip

Finally, we note that the case $d=2$ in the above theorem was proved by Zhang in~\cite{Zhang}, and, was shown by him there to have some very interesting applications in arithmetic geometry (see also~\cite{Cinkir}). 
For a  general approach to non-Archimedean Arakelov geometry using Berkovich theory and tropical geometry, see~\cite{CLD, GK}.

  \subsection{Organization of the paper} In Section~\ref{sec:generalities}, we give the 
  formal definition of the simplicial structure of $\G$, and 
  prove several basic properties of the Chow ring which will be used all through the paper. 
  The structure theorem is proved in Section~\ref{sec:str}. The proof of the localization theorem is given in Section~\ref{sec:loc}. 
In Section~\ref{sec:hypercube}, we prove Theorem~\ref{thm:main1-intro}. 
Section~\ref{sec:fourier} is devoted to the study of the structure of the Chow ring in the Fourier dual 
basis. In particular, 
the vanishing Theorem~\ref{thm:vanishing-intro} is proved in that section.

      \section{Basic definitions and properties} \label{sec:generalities}
  
 In this section, we define the simplicial set structure on products of graphs, and prove basic results on the structure of the combinatorial Chow ring.

\medskip

All through this section, by $G_1 = (V_1, E_1), \dots, G_d = (V_d, E_d)$ we denote $d$ simple connected graphs. All graphs are finite. 

\subsection{Simplicial set structure on the product of graphs} 
We view $G_i$ as a simplicial set of dimension one in a natural way. Suppose that for each $i=1, \dots, d$, a total order $\leq_{G_i}$, 
or simply $\leq_i$ if there is no risk of confusion, on the vertices of $G_i$ is fixed. We can endow 
 the product $\G := G_1 \times\dots \times  G_d$ with a simplicial set structure of dimension $d$ induced by orders 
 $\leq_i$. This works as follows. The set of vertices (0-simplices) of $\G$  is $\G_0 = V_1\times \dots \times V_d$.
 For two vertices  $\v = (v_1,\dots, v_d)$ and $\u=(u_1,\dots, u_d)$ in $\G_0$, we say $\u\leq \v$ if for any $i=1,\dots, d$, we have $u_i \leq_i v_i$.
  A 1-simplex of $\G$ is a pair of vertices $\v = (v_1,\dots, v_d)$ and $\u=(u_1,\dots, u_d)$ in $\G_0$ such that $\v\leq \u$ and such that, in addition, for each $1\leq i\leq d$, either $v_i = u_i$, or if 
  $v_i <_i u_i$, then $\{v_i,u_i\}$ is an edge of $G_i$. The set of 1-simplices of $\G$ is denoted by $\G_1$. An element $\{\u,\v\} \in \G_1$ as above is non-degenerate if $\u \neq \v$. 
 For any 1-simplex $\{\u,\v\}$ with $\u\leq \v$, denote by 
  $I(\u,\v)$ the set of all indices $i\in \{1,\dots, d\}$ with $u_i <v_i$. In particular, if $\u = \v$, we have $I(\u,\v)=\emptyset$.

   More generally, for $k\in \mathbb N$, the set $\G_k$ of $k$-simplices of $\G$ is defined as follows. 
   A $k$-simplex $\sigma$ is a sequence $\v_0\leq \dots \leq \v_k$
of vertices in $\G_0$  such that for each $0\leq j\leq k-1$, the pair $\{\v_j,\v_{j+1}\}$ belongs to 
$\G_1$, and in addition, the sets $I(\v_j, \v_{j+1})$ are all pairwise disjoint. 
We denote by $I(\sigma)$ the union of all the disjoint sets $I(\v_j, \v_{j+1})$, for $0\leq j\leq k-1$. 

\noindent We say $\sigma$ is non-degenerate if we have $\v_0<\v_1<\dots <\v_k$. 
The set of non-degenerate $k$-simplices of $\G$ is denoted by $\G^{nd}_k$. 

 \medskip
 
 Here is an alternative way to describe to simplicial structure of $\G$. First, for each $1\leq i\leq d$, we orient the edges of $G_i$ with respect to the total order $\leq_i$ in such a way that any edge $\{u,v\} \in E_i$ gets orientation $uv$ with $u <_i v$. By an abuse of the notation, we use $E_i$ to denote as well. the set of oriented edges of $G_i$ given by the total order $\leq_i$. 

\medskip

Let $\E = E_1\times \dots\times E_d$, and for each $\e=(e_1,\dots, e_d) \in \E$, for oriented edges 
 $e_1\in E_1, \dots, e_d\in E_d$, denote by $\square_\e$ the product $e_1 \times \dots \times e_d$. We identify $\square_\e$ with the $d$-dimensional cube $\square^d$ with vertices $\{0,1\}^d$ via the identification 
 of each 
 oriented edge  $e_i = u_iv_i$ with $\{0,1\}$, identifying thus $u_i$ with $0$ and $v_i$ with 1. We endow the hypercube $\square^d$ with its standard simplicial structure. Namely, identify $\square^d$ with the vertex set of the hypercube $[0,1]^d$, and for each element $\sigma$ of the symmetric group $\mathfrak S_d$ of order $d$, define \[\Delta_\sigma :=\Bigl\{\,(x_1,\dots, x_d) \in [0,1]^d \,\Bigl |\, \,0\leq x_{\sigma(1)} \leq \dots \leq x_{\sigma(d)}\leq 1
 \,\Bigr\}.\]
 The non-degenerate  $d$-simplices of $\square^d$ are the vertices of $\Delta_\sigma$, for any element $\sigma \in \mathfrak S_d$.
 The simplicial structure on $\G$ is the union of the simplicial structures on each $\square_\e \simeq \square^d$ for $\e\in \E$. 
 
 \medskip
 
 \noindent {\bf Notation.} All through the paper, we use bold letters $\u,\v,\w, \textrm{ etc.}$ to denote a 0-simplex in a product of graphs. For graphs $G_1, \dots, G_d$ with product $\G$, if $\v,\u,\w, \textrm{etc.}$ is a vertex in $\G_0$, we denote by $v_i,u_i,w_i, \textrm{etc.}$, respectively, to denote the corresponding vertex of the graph $G_i$, so we have $\v = (v_1, \dots, v_d), \u=(u_1,\dots, u_d), \textrm{ etc.}$
 \subsection{Definition of the combinatorial Chow ring}  We recall the definition of the Chow ring given in the introduction, and use the opportunity to introduce a few useful notations.  Denote by $Z(\G)$ the free polynomial ring with coefficients in $\Z$ generated by the vertices of $\G$, namely, 
 \[Z(\G) := \Z[C_{\v} \,|\, \v\in \G_0],\]
 where the variables $C_{\v}$ are associated to the vertices in $\G_0$.
 We view $Z(\G)$ as a graded ring where each variable $C_\v$ is of degree one. For $k\in \mathbb N$, 
 denote by $Z^k(\G)$ the  graded piece consisting of homogeneous polynomials of degree $k$.
 
 Let $\mathscr I_{\mathrm{rat}}(\G)$, or simply $\mathscr I_\rat$ if there is no risk of confusion, be the  graded ideal of all the elements of $Z(\G)$ which are rationally equivalent to zero: this is the ideal generated by the following generators of homogenous degrees two, two, and three, respectively. 
  \begin{itemize}
  \item[($\mathscr R1$)] $C_{\v_1} C_{\v_2} \dots C_{\v_k}$ \qquad for $k\in \N$  and elements $\v_j\in \G_0$ such that $\v_1,\dots, \v_k$ do not form a simplex in $\G$\,; 
  \item[($\mathscr R2$)]  $ C_{\u} \Bigl(\,\sum_{\v \in \G_0} C_{\v}\,\Bigr)$\qquad for any vertex  $\u \in \G_0$; and
  \item[($\mathscr R3$)]   $C_{\u} C_{\mathbf{w}}\Bigl(\,\sum_{\v \in \G_0: v_i = u_i}C_{\v}\,\Bigr)\, $ \quad for any pair of vertices $\u, \mathbf{w} \in \G_0$ and any index $1\leq i \leq d$ with $u_i \neq w_i$.
 \end{itemize}
  For two elements $\alpha, \beta \in Z(\G)$, we write $\alpha \sim_{\mathrm{rat}} \beta$ iff $\alpha- \beta \in \mathscr I_{\rat}$. 
 
 \medskip

 The  combinatorial Chow ring of $\G$ is the graded ring $\Ch(\G) := Z(\G)/\mathscr I_{\rat}$. It has a natural grading, and for $k\in \mathbb N$, we denote by $\Ch^k(\G)$ the graded  piece of degree $k$.

  \begin{remark}\rm
   There are other types of cohomological rings one can associate to a product of graphs, e.g., the Stanley ring of the product of graphs (with its natural cubical structure)~\cite{Het}, 
   the tropical Chow ring of products of (metric) graphs~\cite{AR, Shaw}, Tropical cohomology groups~\cite{IKMZ}, and the Chow ring of matroids~\cite{AHK}.   
    It would be interesting to study the relation between these different rings.  
  \end{remark}

\begin{remark}\rm 
As it was mentioned before, the Chow ring $\Ch(\G)$ comes with a map $\alpha_\fX: \Ch(\G) \to \mathrm{Chow}_{\fX_s}(\fX)$ for the Gross-Schoen desingularization  $\fX$ of a product of 
regular proper semi-stable curves $\fX_1, \dots, \fX_d$ over discrete valuation ring $R$ where the dual graph of the special fiber of each $\fX_{i,s}$ is $G_i$. 
It seems natural to expect that under some genericity condition on the semi-stable curves $\fX_{i,s}$, and the regular smoothings $\fX_i$ of $\fX_{i,s}$,  the ring $\Ch(\G)$ becomes isomorphic to the subring $\mathrm{Chow}^c_{\fX_s}(\fX)$ of the Chow ring $\mathrm{Chow}_{\fX_s}(\fX)$ generated by the irreducible components of the special fiber of $\fX$.  \end{remark}
 
 It will be useful to introduce the following.
 \begin{defi}[The ideal $\mathscr I_1$]\rm
 Denote by $\mathscr I_1$ the ideal of 
the polynomial ring 
$Z(\G) = \Z[C_\v|\v\in \G_0]$ generated by the relations $(\mathscr R1)$, i.e., by the products 
$C_{\v_1}\dots C_{\v_k}$ for any $k\in \mathbb N$ and $\v_1,\dots, \v_k\in \G_0$ which do not form a simplex.
\end{defi}

 \subsection{Functoriality.} 
 Let $H_1,\dots, H_d$ be $d$ simple connected graphs with orders $\leq_{H_i}$ on the vertices of each $H_i$. Define $\mathscr H= H_1\times \dots \times H_d$ with the induced simplicial structure as described above. Suppose that for each $i=1, \dots, d$, a homomorphism of graphs 
 $f_i : H_i \rightarrow G_i$ is given such that $f_i$ respects also the two orderings $\leq_{H_i}$ and $\leq_{G_i}$, namely, for two vertices $u\leq_{H_i} v$ of $H_i$, we have $f(u) \leq_{G_i} f(v)$ in $G_i$. 
By the definition of the simplicial structure, the product of $f_i$ leads to 
 a morphism of simplicial sets 
 $f: \mathscr H \rightarrow \G$. Moreover, $f$ induces the graded algebra map $f^* :  Z(\G) \rightarrow Z(\H)$, which is defined on the level of generators by sending $C_\v$ for $\v \in \G_0$ to
 \[f^*(C_\v) = \sum_{\substack{\u\in \mathscr H_0\\ f(\u) =\v}} C_\u.\]
\begin{prop} \label{prop:func} Notations as above, the map $f^*$ sends $\Rat(\G)$ to $\Rat(\H)$ and induces a map of Chow rings $f^*: \Ch(\G) \rightarrow \Ch(\H)$.
\end{prop} 
\begin{proof} We need to prove that generators of $\Rat(\G)$ given by  ($\mathscr R1$), ($\mathscr R2$) and ($\mathscr R3$) are sent to $\Rat(\H)$.

 Let $k\in \mathbb N$ and $\v_1, \dots, \v_k \in \G_0$ such that $\v_1, \dots, \v_k$ do not form a simplex in $\G$.  Since $f : \H \to \G$ is a map of simplicial sets, it follows for any set of vertices $\u_1, \dots, \u_k \in \H_0$ with $f(\u_j) = \v_j$ for $j=1, \dots, k$, the vertices $\u_1, \dots, \u_k$ do not form a simplex in $\H$.  It follows that $f^*(C_{\v_1}) f^*(C_{\v_2}) \dots f^*(C_{\v_k})  \in \Rat(\H)$.

 Let now $\v \in \G_0$. We have 
 \begin{align*}
 f^*\Bigl(\,C_\v \bigl(\sum_{\mathbf w\in \G_0} C_{\mathbf w}\bigr) \Bigr) &= f^*(C_\v)  \Bigl(\sum_{\mathbf w\in \G_0} f^*\bigl(C_{\mathbf w}\bigr) \Bigr) = f^*(C_\v) \Bigl(\sum_{\mathbf x\in \H_0} C_{\mathbf x}\Bigr)  \in \Rat(\H).
 \end{align*}
 
Finally, let $\v,\w \in \G_0$ and $i\in \{1,\dots, d\}$ such that $v_i \neq w_i$. We have 
\begin{align*}
f^*\Bigl (C_{\v} C_\w \bigl(\,\sum_{\mathbf z \in \G_0: z_i = v_i} C_{\mathbf z}\,\bigr)\Bigr) &= f^*(C_\v) f^*(C_\w) \bigl(\,\sum_{\mathbf z \in \G_0: z_i = v_i} f^*(C_{\mathbf z})\,\bigr) \\
&= \sum_{\substack{\u \in \H_0\\ 
f(\u) =\v}}\sum_{\substack{\mathbf x \in \H_0 \\
f(\mathbf x) =\w}} C_\u C_{\mathbf x} \sum_{\substack{\mathbf y \in \H_0 \\ f_i(y_i) = v_i}} C_\mathbf y\\
&= \Bigl(\,\sum_{\substack{\u \in \H_0\\ 
f(\u) =\v}}\sum_{\substack{\mathbf x \in \H_0 \\
f(\mathbf x) =\w}} C_\u C_{\mathbf x} \sum_{\substack{\mathbf y \in \H_0 \\ y_i = u_i}} C_\mathbf y \, \Bigr) \\
&\qquad + \Bigl(\,\sum_{\substack{\u \in \H_0\\ 
f(\u) =\v}}\sum_{\substack{\mathbf x \in \H_0 \\
f(\mathbf x) =\w}} C_\u C_{\mathbf x} \sum_{\substack{\mathbf y \in \H_0 \\ f_i(y_i) = v_i \\ y_i \neq u_i}} C_\mathbf y \Bigr).
\end{align*}
For $\u, \mathbf x\in \H_0$ with $f(\u) = \v$ and $f(\mathbf x) =\w$, we have $u_i \neq x_i$.
Therefore, we have 
$C_\u C_{\mathbf x} \sum_{\substack{\mathbf y \in \H_0 \\ y_i =u_i}}C_\mathbf y \in \Rat(\H)$.  

In addition, for such $\u, \mathbf x \in \H_0$, and for any $\mathbf y\in \H_0$ with $f_i(y_i)=v_i$ and $y_i\neq u_i$, since $v_i \neq w_i$, we have  $y_i \neq x_i$. Thus, $C_\u C_{\mathbf x} C_{\mathbf y}$ is not a simplex in $\H$, and we have $C_\u C_{\mathbf x} C_{\mathbf y} \in \Rat(\H)$. This shows that 
$f^*\Bigl (C_{\v} C_\w \bigl(\,\sum_{\mathbf z \in \G_0: z_i = v_i} C_{\mathbf z}\,\bigr)\Bigr) \in \Rat(\H)$, and the proposition follows.
\end{proof}

\subsubsection{Permutation of factors}
Let $\sigma\in \mathfrak S_d$ be an element of the permutation group of order $d$. Given simple graphs $G_1, G_2, \dots, G_d$, define $\mathscr G_\sigma := G_{\sigma(1)}\times \dots \times G_{\sigma(d)}$, and denote by $\mathscr V_\sigma$ its vertex set. For any vertex $\v = (v_1, \dots, v_d) \in \mathscr V$, let $\v_\sigma:=(v_{\sigma(1)}, \dots, v_{\sigma(d)})$. We have an isomorphic of polynomial rings $\eta_\sigma : Z(\G) \rightarrow Z(\G_\sigma)$ which sends the generator $C_{\v}$ to $ C_{\v_\sigma}$. The following proposition is immediate. 
\begin{prop}\label{prop:permut} Notations as above, the map $\eta_\sigma$ induces an isomorphism of Chow rings $\Ch(\mathscr G) \rightarrow \Ch(\mathscr G_\sigma))$. 
\end{prop}

\subsection{Intersection maps on the level of Chow groups} 
For a graph $G=(V,E)$ with a total order $\leq$ on its vertex set $V$, and for any vertex $v$, we denote by $G[\leq v]$ (resp. $G[<v]$) the induced graph on the set of vertices $\{u\,|\, u\leq v\}$ (resp. $\{u\,|\, u<v\}$). Let $G_1, \dots, G_d$ be simples graphs, and let $\v = (v_1, \dots, v_d)\in \mathscr V$ with $v_i\in V(G_i)$, for $i=1, \dots, d$. For each $1\leq i\leq d$, define 
$$H_i:= \begin{cases} G_i[\leq_i v_i] \qquad \textrm{if $i\neq k$}\\
G_k[<v_k] \qquad \textrm{if $i=k$},
\end{cases}$$
and set 
$\mathscr G_{\v,k} := H_1 \times \dots \times H_d.$ Denote by $\mathscr{V}_{\v, k}$ the set of vertices of $\G_{\v,k}$.

For each $i$, we have an inclusion $V(H_i) \subseteq V(G_i)$, which induces an inclusion $\mathscr V_{\v, k} \subseteq \mathscr V$.  Total orders $\leq_i$ induced total orders on the vertex set of each $H_i$, 
from which $\mathscr G_{\v,k}$ inherits a simplicial structure, and the inclusion respects the simplicial structures. Thus, we can write $\mathscr G_{\v,k} \subseteq \mathscr G$ as simplicial sets.

Consider the map of $\Z$-modules $\beta=\beta_{\v,k}: \Z[\mathscr V_{\v,k}] \rightarrow \Z[\mathscr V]$ defined by multiplication by $C_\v$
\[\forall\,i \in \mathbb N\,\,\,\forall\, \w_1, \dots, \w_i\in \mathscr V_{\v,k},\qquad \beta\bigl(C_{\w_1} C_{\w_2} \dots C_{\w_i}\bigr) := C_{\w_1} C_{\w_2}\dots C_{\w_i}C_\v.\]
We have 
\begin{prop}\label{prop:intersection}
 The map $\beta$ induces a well-defined map of $\Z$-modules $\beta: \Ch(\mathscr G_{\v,k}) \rightarrow \Ch(\mathscr G)$.
\end{prop}

\begin{proof}
 We will prove the three set of relations $(\mathscr R1), (\mathscr R2),(\mathscr R3)$ defining $\Ch(\mathscr G_{\v,k})$ vanish by $\beta$ in $\Ch(\mathscr G)$, from which the result follows.

Using Proposition~\ref{prop:permut}, and permuting factors if necessary, we can without loss of generality assume that $k=d$.

\noindent $\bullet \,(\mathscr R1)$ If $\w_1, \dots, \w_i\in \mathscr V_{\v,d}$ do not form a simplex in $\mathscr G_{\v,d}$, then obviously, they do not form a simplex in $\mathscr G$, and we have 
  \[\beta(C_{\w_1} \dots C_{\w_i}) = C_{\w_1} \dots C_{\w_i}C_\v=0 \textrm{  in $\Ch(\mathscr G)$.}\]

\noindent  $\bullet \,(\mathscr R2)$ We show that for any $\u \in \mathscr V_{\v,d}$, we have 
 \[\beta\bigl( C_\u\sum_{\w\in \mathscr V_{\v,d}} C_\w \bigr)=0 \qquad \textrm{in }\, \Ch(\mathscr G).\]
We have 
 \begin{align*}
 \beta \bigl(C_\u \sum_{\w\in \mathscr V_{\v,d}} C_\w\bigr)  &=  C_{\u}  \bigl(\sum_{\w\in \mathscr V_{\v,d}} C_{\w}\bigr) C_\v\\
 &= C_{\u}\bigl(\sum_{\substack{{\bf z}\in \mathscr V\\ {\bf z\leq v} \,,\, z_{d}< v_d}} C_{\bf z}\bigr)  C_\v  \\
 &= C_{\u}\bigl(\sum_{\substack{{\bf z}\in \mathscr V\\ {\bf z\leq v} \,,\, z_{d} = u_d}} C_{\bf z}\bigr)  C_\v  \textrm{ by $(\mathscr R1)$ since $u_d < v_d$ and $z_d <v_d$}\\
 &= \bigl(\sum_{\substack{{\bf z}\in \mathscr V\\ z_{d}=u_d}} C_{\bf z}\bigr) C_{\u} C_\v \,\,-\,\, \bigl(\sum_{\substack{{\bf z}\in \mathscr V\\ {\bf z \not \leq \v}\,,\,z_{d}=u_d <v_d}}C_{\bf z}\bigr) C_{\u} C_\v \\
 &=0\,\,\in \Ch(\mathscr G).
\end{align*}
In the last equation above, we have used  the vanishing of the first term $\bigl(\sum_{\substack{{\bf z}\in \mathscr V \\ z_{d}=u_d}} C_{\bf z}\bigr) C_{\u} C_\v =0$,
implied by $(\mathscr R2)$  since $u_d\neq v_d$, and the vanishing for any $\bf z\not \leq v$ with $z_{d}=u_d<v_d$ of the product $C_{\bf z} C_{\u} C_\v$ , since in this case, 
$\bf z$ and $\v$ cannot form a simplex.

\medskip

\noindent $\bullet \,(\mathscr R3)$ We have to show that  for all $j \in \{1,\dots, d\}$ and any $\u,\w \in \mathscr V_{\v,d}$ with  $u_j\neq w_j$, we have  %
 \[\beta\bigl(C_\u C_\w\sum_{\substack{{\bf z}\in \mathscr V_{\v,d} \\ z_j=u_j}} C_{\bf z}\,\bigr) =0.\]
By $(\mathscr R1)$, we have $C_{\u}C_{\w}C_{\mathbf z} =0$ for any $\mathbf z$ with $z_{j} \neq u_j, w_j$. It follows that
$$C_\u C_\w\sum_{\substack{{\bf z}\in \mathscr V_{\v,d}}} C_{\bf z} = C_\u C_\w\sum_{\substack{{\bf z}\in \mathscr V_{\v,d} \\ z_j=u_j}} C_{\bf z} + C_\u C_\w\sum_{\substack{{\bf z}\in \mathscr V_{\v,d} \\ z_j=w_j}} C_{\bf z}.$$
 Since $\beta \bigl(C_\u C_\w \sum_{{\bf z}\in \mathscr V_{\v,d}} C_{\bf z}\bigr) =0$ in $\Ch(\mathscr G)$, by $(\mathscr R2)$ that we just proved, we can assume further that $u_j < v_j$. 
 
 If $j=d$, then since $u_d, w_d < v_d$, and $u_d\neq w_d$, we have $\beta(C_\u C_\w)=C_\u C_\w C_\v =0$ by $(\mathscr R1)$ in $\Ch(\mathscr G)$, which directly gives the assertion. 
 
 So we can assume that $j\neq d$. 
 We have 
 \begin{align*}
  \beta\bigl(C_\u C_\w \sum_{\substack{{\bf z}\in \mathscr V_{\v,d}\\ z_j=u_j}} C_{\bf z}\,\bigr) =& C_{\u} C_{\w}\Bigl(\,\sum_{\substack{{\bf z}\in \mathscr V_{\v,d}\\ z_j=u_j}} C_{\bf z}\,\Bigr)C_\v\\
  =&C_{\u} C_{\w}\Bigl(\,\sum_{\substack{{\bf z}\in \mathscr V \\ z_j=u_j}} C_{\bf z}\,\Bigr)C_\v\,-\,C_{\u} C_{\w}\Bigl(\,\sum_{\substack{{\bf z}\in \mathscr V\\ z_{j}=u_{j}\\ {\bf z}\not \leq \v,\, z_{d}<v_d}} C_{\bf z}\,\Bigr)C_\v\\
  &-\,\,C_{\u} C_{\w}\Bigl(\,\sum_{\substack{{\bf z}\in \mathscr V\\ z_{j}=u_{j}\\ {\bf z}\not \leq \v,\, z_{d}=v_d}} C_{\bf z}\,\Bigr)C_\v-\,\,C_{\u} C_{\w}\Bigl(\,\sum_{\substack{{\bf z}\in \mathscr V\\ z_{j}=u_{j}\\ {\bf x}\not \leq \v,\, z_{d}>v_d}} C_{\bf z}\,\Bigr)C_\v.
 \end{align*}
Since $u_{j} < w_{j}$, by $(\mathscr R3)$ in $\Ch(\mathscr G)$, the first term in the above sum vanishes, i.e., 
\[C_{\u} C_{\w} \Bigl(\,\sum_{\substack{{\bf x}\in \mathscr V\\ x_{j}=u_{j}}} C_{\bf x}\,\Bigr)C_\v = 0.\]

For ${\bf x}\not \leq \v$ with $x_{d} < v_d$, ${\bf x}$ and $\v$ do not form a simplex, and so the second term in the sum is also zero, i.e., 
\[C_{\u} C_{\w}\Bigl(\,\sum_{\substack{{\bf z}\in \mathscr V\\ z_{j}=u_{j}\\ {\bf z}\not \leq \v,\, z_{d}<v_d}} C_{\bf z}\,\Bigr)C_\v = 0.\]

Let now ${\bf z}\in \mathscr V$ with ${\bf z} \not \leq \v$ and $z_{d}=v_d$. Since $z_{j} = u_j <w_j$ and $z_d=v_d>w_d$, we infer that $\bf z$ and $\w$ do not form a simplex, and the third term in the sum vanishes as well, i.e., 
\[C_{\u} C_{\w}\Bigl(\,\sum_{\substack{{\bf z}\in \mathscr V\\ z_{j}=u_{j}\\ {\bf z}\not \leq \v,\, z_{d}=v_d}} C_{\bf z}\,\Bigr)C_\v = 0.\]
As for the last term, we have $z_j=u_j < v_j$ and $z_d>v_d$, so $\mathbf z$ and $\v$ do not form a simplex, which gives 
\[C_{\u} C_{\w}\Bigl(\,\sum_{\substack{{\bf z}\in \mathscr V\\ z_{j}=u_{j}\\ {\bf x}\not \leq \v,\, z_{d}>v_d}} C_{\bf z}\,\Bigr)C_\v =0.\]
Combining all this, we thus get $\beta\bigl(C_\u C_\w \sum_{\substack{{\bf z}\in \mathscr V_{\v,d}\\ z_j=u_j}} C_{\bf z}\,\bigr) =0,$ and the proposition follows.
\end{proof}

 \subsection{Moving lemma}
  The moving lemma for the  Chow ring~\cite{Kolb1} is the statement that each 
  graded piece $\Ch^{k+1}(\G)$ is generated by monomials of the form $\prod_{\v\in \sigma} C_{\v}$, 
  for $\sigma \in \G^{nd}_k$. 
  \begin{thm}\label{thm:nd} For any $k\in \mathbb N$, the  Chow group $\Ch^{k+1}(\G)$ is generated by 
   monomials of the form $C_\sigma=C_{\v_0} \dots C_{\v_k}$, where $\sigma$ is a non-degenerate 
   simplex of dimension $k$ in $\G$ with vertex set $\v_0,\dots, \v_k$.
  \end{thm}
 
 We give a proof of this theorem based on  Proposition~\ref{prop:rewr} below,
 which will be crucial later in the proof of the structure Theorem~\ref{thm:st}.

 \medskip
 
 First we introduce some terminology. For any $k$-simplex $\tau$ of $\G$ with at least two different 
 vertices,
 and for $i\in I(\tau)$ and $\epsilon \in \{0,1\}$, let $\{u_i,v_i\}$ be the corresponding edge of $G_i$ with $u_i <_i v_i$, and 
 define the element 
 $R^{\epsilon}_{\tau,i}$ of $Z(\G)$ by
 \[R^0_{\tau,i}:= \sum_{\substack{\w \in \G_0 \\ w_i =u_i}} C_{\w}, \qquad \textrm{and} \qquad  
 R^1_{\tau,i}:= \sum_{\substack{\w \in \G_0 \\ w_i =v_i}} C_{\w}.\]

 Note that the product $C_\tau R^\epsilon_{\tau,i}$ is among the relations $(\mathscr R3)$ and thus belongs to 
 $ \mathscr I_\rat$.

 \begin{defi}\rm Let $k,m$ be two natural numbers, and let $\tau \in \G_k^{nd}$ be a non-degenerate $k$-simplex. 
 Define the set 
 $\mathcal A_{\tau}(m)$ as the collection of all the multisets $S$ of size $m$ consisting of 
 $m$ (possibly equal) 
 elements $(i_1, \epsilon_1), 
 \dots, (i_m,\epsilon_m)$ with $i_1,\dots, i_m\in I(\tau)$ and 
 $\epsilon_1, \dots, \epsilon_m \in \{0,1\}$. 
 \end{defi}

 With these notations, we have the following useful proposition.
  \begin{prop}\label{prop:rewr}
Let $\sigma \in \G_k$ be a simplex with vertices $\v_0\leq \dots \leq \v_k$ and with at least two distinct 
vertices.   There exist an element $\beta \in \mathscr I_1$, and a collection of integers $a_{\tau,S} \in \mathbb Z$ for any $1\leq l \leq k$, any $\tau \in \G_l^{nd}$ such that 
$\sigma \subseteq \tau$, and any $S \in \mathcal A_\tau(k-l)$, such that we have
\[C_\sigma = \beta+\sum_{l=1}^k \sum_{\substack{\tau \in \G_{l}^{nd}\\ {\textrm s.t. }\sigma \subseteq \tau, \\ S \in \mathcal A_\tau(k-l)}} a_{\tau,S}\, C_\tau 
\prod_{(i,\epsilon) \in S} R_{\tau,i}^\epsilon.\]
 \end{prop}

\begin{proof}[Proof of Theorem~\ref{thm:nd}] By Proposition~\ref{prop:rewr}, for any simplex 
$\sigma \in \G_k$ with at least two different vertices, we can write 
\[C_\sigma = \beta + \sum_{l=1}^k \sum_{\substack{\tau \in \G_{l}^{nd} \\ S \in \mathcal A_\tau(k-l)}} a_{\tau,S}
\, C_\tau \prod_{(i,\epsilon) \in X} R_{\tau,i}^\epsilon,\]
for integers $a_{\tau,S}$ and $\beta \in \mathscr I_1 \subset \mathscr I_\rat$. For $l\neq k$, and any $S \in \mathcal A_\tau(k-l)$, the term 
$C_\tau \prod_{(i,\epsilon) \in S} R^\epsilon_{\tau,i}$ belongs to $\mathscr I_\rat$. It follows that 
\[C_\sigma \sim_\rat \sum_{\tau \in G_k^{nd}} a_{\tau, \emptyset} C_\tau,.\]
Also note that for any integer $k\geq 2$ and any $\u\in \G_0$, we have by
$(\mathscr R2)$
\[C_\u^k \sim_\rat -\sum_{\substack{\v\in \G_0\\ \{\u,\v\}\in \G_1^{nd}}} C_\v C_{\u}^{k-1}, \]
and so applying the previous case, it follows that all the monomials of degree $k$ in $Z(\G)$ are rationally equivalent to an integral linear combination of the monomials $C_\tau$ for 
$\tau \in \G_{k}^{nd}$, from which the theorem follows. 
\end{proof}

\begin{proof}[Proof of Proposition~\ref{prop:rewr}] The proof goes by induction on $k$. 
For the base case $k=1$, note that any 1-simplex with at least two distinct vertices is necessarily non-degenerate, and so the result trivially holds in this case. Let $k\geq 2$ be an integer, and assume the result holds for all $k'$-simplices with at least two distinct 
vertices for any 
$k'<k$.   We prove it holds  as well for any simplex $\sigma \in \G_k$ with at least two distinct vertices. 

We proceed by a reverse induction on the number of different vertices of $\sigma$. If $\sigma$ is non-degenerate, i.e., if it has 
$k+1$ distinct vertices, the result is obvious. Suppose that $2\leq l <k$, and the result holds for all 
simplices in $\G_k$ with at least 
$l+1$ distinct vertices. Let $\sigma \in \G_k$ with vertex set $\v_0\leq \v_1\leq \dots \leq \v_k$ such that the set 
$\{\v_0,\dots, \v_k\}$ is of size $l$. 
We prove the result for $\sigma$. Denote by $\u_1<\u_2<\dots <\u_l$ all the different vertices of $\sigma$, and by $n_1, \dots, n_l$ the multiplicity of 
$\u_1, \dots, \u_l$ in $\sigma$, respectively. (I.e., the number of times each $\u_j$ appears among the vertices $\v_0,\dots, \v_k$ of $\sigma$.) 
We proceed by (a third) induction on the lexicographical order on ordered sequences $(n_1, \dots, n_l)$. 
Recall that for two ordered sequences ${\mathbf m} = (m_1, \dots, m_l)$ and ${\mathbf n}=(n_1, \dots, n_l)$, 
we have ${\mathbf m} \geq_{lex} {\mathbf n}$ if there exists $0\leq s\leq l$ such that $m_{s+1}> n_{s+1}$, 
and $m_{t}\geq n_{t}$ for all $t\leq s$.

Consider first the smallest ordered sequence $(n_1, \dots, n_l)$ in the lexicographical order, so that 
we have $n_1 =\dots=n_{l-1}=1$, and $n_l = k-l+1>1$ (since $l<k$).  Let $i\in I(\u_{l-1}, \u_l)$. There exists an element $\beta_0 \in \mathscr I_1$ such that we have
\begin{align*}
 C_{\sigma} =& \beta_0+ C_{\u_1}\dots C_{\u_{l-1}}C_{\u_l}^{k-l+1} \\
 =& \beta_0+C_{\u_1}\dots C_{\u_{l-1}}C_{\u_l}^{n_l-1}
R_{\sigma,i}^1 - \sum_{\substack{\w\in \G_0\\
\w > u_l}} C_{\u_1} \dots C_{\u_{l-1}} C_{\u_l}^{n_l-1} C_{\w} \\
&- \sum_{\substack{\w\in \G_0 \\ \u_{l-1}<\w < \u_l\\
w_i = u_{l,i}}}C_{\u_1} \dots C_{\u_{l-1}}C_{\w} C_{\u_l}^{n_l-1}.
\end{align*}
Each term $C_{\u_1} \dots C_{\u_{l-1}} C_{\u_l}^{n_l-1} C_\w$ in the above sum either belongs to $\mathscr I_1$ or is of the form $C_\tau$ for a $k$-simplex $\tau$ which has
$l+1$ different vertices. Also the term $C_{\u_1}C_{\u_2} \dots C_{\u_{l}}^{n_l-1}$ is $C_\tau$ for a $(k-1)$-simplex $\tau$ with at least two distinct vertices. Thus the result  follows by applying the induction hypothesis to  each term appearing in the right hand side of the above equation.  

By symmetry 
the same reasoning applies to the maximum ordered sequence $(n_1, \dots, n_l)$ in the lexicographical order 
which has $n_2 =\dots =n_l=1$. 

Let now ${\mathbf n}=(n_1, \dots, n_l)$ be an arbitrary ordered sequence. We can assume that 
$\mathbf n$ is neither 
maximum nor minimum in the lexicographical order. Thus, there exists $1<h\leq l$ such that $n_h \geq 2$. Let $i\in I(\u_{h-1}, \u_h)$. 
Quite similarly as above, there exists $\beta_1 \in \mathscr I_1$ such that we have 

\begin{align*}
 C_{\sigma} = C_{\u_1}^{n_1}\dots C_{\u_{h}}^{n_h}\dots C_{\u_l}^{n_l} 
 =& \beta_1+ C_{\u_1}^{n_1}\dots C_{\u_{h-1}}^{n_{h-1}} C_{\u_{h}}^{n_h-1}\dots C_{\u_l}^{n_l}
R_{\sigma,i}^1 \\
&- \sum_{\substack{\w\in \G_0\\ \w > u_h}} C_{\u_1}^{n_1} \dots C_{\u_{h-1}}^{n_{h-1}}
C_{\u_{h}}^{n_h-1}\dots C_{\u_l}^{n_l} C_{\w} \\
&-\sum_{\substack{\w\in \G_0\\
\u_{h-1}<\w \leq u_h\\ w_i = u_{h,i}}} C_{\u_1}^{n_1} \dots 
C_{\u_{h-1}}^{n_{h-1}}C_{\u_{h}}^{n_h-1}\dots C_{\u_l}^{n_l} C_{\w}.
\end{align*}
The hypothesis of the induction applies to the first term as 
$C_{\u_1}^{n_1}\dots C_{\u_{h-1}}^{n_{h-1}} C_{\u_{h}}^{n_h-1}\dots C_{\u_l}^{n_l}$ has degree $k-1$. 
In the second term of the equation above, the induction hypothesis applies to each term 
$C_{\u_1}^{n_1} \dots C_{\u_{h-1}}^{n_{h-1}}
C_{\u_{h}}^{n_h-1}\dots C_{\u_l}^{n_l} C_{\w}$ in the sum: if $\w = \u_{j}$ for some $j\geq h+1$ since the the ordered sequence 
$(n_1, \dots, n_{h-1}, n_h-1, \dots, n_j+1, \dots, n_l)$ is smaller than $\mathbf n$ in the 
lexicographical order. If $\w \neq \u_j$ for all $j$, then the term is either in $\mathscr I_1$ or is of the form $C_{\sigma'}$ with $\sigma'\in \G_k$ with 
more than $l$ distinct vertices. 

Similarly, the hypothesis of the induction applies to each term $C_{\u_1}^{n_1} \dots 
C_{\u_{h-1}}^{n_{h-1}}C_{\u_{h}}^{n_h-1}\dots C_{\u_l}^{n_l} C_{\w}$ in the last sum, as each of those terms has $l+1$ distinct vertices. 
\end{proof}

\section{Proofs of Theorem~\ref{thm:st} and Theorem~\ref{thm1}}\label{sec:str}
By Theorem~\ref{thm:nd}, we have a surjective map 
$ \mathbb Z\langle \G_k^{nd} \rangle \longrightarrow \Ch^{k+1}$. The structure Theorem~\ref{thm:st} describes the kernel of this surjection.  

 First recall that for two elements $\u<\v$ of $\G_0$ such that $\{\u,\v\} \in \G_1$, we denote by $\e_{\u,\v}$ the cube of dimension $|I(\u,\v)|$ formed by all the vertices $\mathbf z=(z_1, \dots, z_d)$ in $\G_0$ with 
  $z_i \in \{u_i,v_i\}$, i.e., $\e_{\u,\v} = \prod_{i=1}^d \{u_i,v_i\}$.  
  
  Let $\sigma$ be a $k$-simplex in $ \G_{k}^{nd}$ with vertices $\v_0<\v_1<\dots<\v_k$. For two indices $1\leq i,j \leq d $ lying both in $I(\u_{t}, \u_{t+1})$ for some $0\leq t \leq r-1$, we defined 
 \[\widetilde R_{\sigma,i,j} := \sum_{\substack{\w\in \e_{\u_t, \u_{t+1}}\\ w_i =u_{t,i}}} C_\w - \sum_{\substack{\w \in \e_{\u_t,\u_{t+1}}\\ w_j =u_{t,j}}} C_\w.\]
 Remark that we have 
\begin{align*}
C_{\sigma} \widetilde R_{\sigma, i,j}  =& \sum_{\substack {\v_t<\w < \v_{t+1}\\\w \in \e_{\u,\v}, w_i=v_{t,i}}} C_{\v_0} \dots C_{\v_t} C_{\w} C_{\v_{t+1}}\dots C_{\v_k} \\
&-\sum_{\substack {\v_t<\w < \v_{t+1}\\\w\in \e_{\u,\v}, w_j=v_{t,j}}} C_{\v_0} \dots C_{\v_t} C_{\w} C_{\v_{t+1}}\dots C_{\v_k}, 
\end{align*}
and so  $C_{\sigma} \widetilde R_{\sigma, i,j} \in \mathbb Z\langle \G_{k}^{nd}\rangle$. Define $\mathscr I_k^{nd}$ as the submodule of $\mathbb Z\langle \G_k^{nd} \rangle$ generated by all the elements $C_\sigma \widetilde R_{\sigma, i,j}$, for any $\sigma, t, i,j$ as above.

  \medskip
  
For $\sigma, t, i,j$ as above, we  define 
\[R_{\sigma,i,j} := R^0_{\sigma,i} - R^0_{\sigma_j} = \sum_{\substack{\w\in \G_0 \\ w_i = v_{t,i}}} C_\w - \sum_{\substack{\w\in \G_0 \\ w_i = v_{t,j}}} C_\w.\]  
Note that we have $C_{\sigma} R_{\sigma, i,j} \sim_\rat 0$ by $(\mathscr R3)$. 

The following proposition is straightforward.
\begin{prop}\label{prop:triv}
Notations as above, there exists $\beta \in \mathscr I_1$ such that we have 
$C_\sigma R_{\sigma, i,j}  = \beta + C_\sigma \widetilde R_{\sigma, i,j}$.  
\end{prop}
This shows that $\mathscr I_k^{nd} \subset \Rat$, and therefore, passing to the quotient, we get a surjection
\[\mathbb Z\langle \G_k^{nd} \rangle/\mathscr I_{k}^{nd} \longrightarrow \Ch^{k+1}.\]
In this section, we prove this map is injective, which implies Theorem~\ref{thm:st}. 

\medskip

Let $\alpha = \sum_{\tau \in \G_k^{nd}} a_\tau C_\tau \in \mathbb Z\langle \G_k^{nd}\rangle$,  with $a_\tau \in \mathbb Z$ for all $\tau \in \G_k^{nd}$,  be an element in the kernel, so we have $\alpha \simeq_\rat 0$. We shall prove that $\alpha \in \mathscr I_{k}^{nd}$. 

\medskip

Consider the graded piece $Z^{k+1}(\G)$ consisting of all polynomials of homogenous degree $k+1$ with integral coefficients in variables $C_\v$ for $\v\in \G_0$. We define on $Z^{k+1}(\G)$ a decreasing filtration $\mathcal F^{\bullet}:\,\, \mathcal F^{-1} = Z^{k+1}(\G) \supset \mathcal F^0 \supset \dots \supset \mathcal F^{k-1} \supset \mathcal F^k$  as follows. 

\begin{defi}\rm Define $\mathcal F^{-1 }:= Z^{k+1}(\G)$, $\mathcal F^{0 }:= \Rat^{k+1}(\G)$, and for each $ 1 \leq l \leq k-1$, define $\mathcal F^{l}$ as the set of all elements $\alpha$ which verify the following property: there exists an element $\beta \in \mathscr I_1$, and for each $l\leq t\leq k-1$, there are integers $a_{\tau, S}$ associated to any $\tau \in \G_{t}^{nd}$ and $S\in \mathcal A_\tau(k-t)$ such that we have 
\[\alpha =\beta+  \sum_{t=l}^{k-1} \sum_{\substack{\tau \in \G_{t}^{nd} \\ S \in \mathcal A_\tau(k-t)}} 
a_{\tau,S}\, C_\tau 
\prod_{(i,\epsilon) \in S} R_{\tau,i}^\epsilon.\]
Finally, define $\mathcal F^{k} := \mathscr I_k^{nd}$. 
\end{defi}
(Note that the inclusion $\mathcal F^0 \subset \mathcal F^1$ is implied from Proposition~\ref{prop:triv}.)

\medskip

Theorem~\ref{thm:st} is a consequence of the following two lemmas.
 
\begin{lemma}\label{lem:crucial} 
Let $\alpha \in \mathbb Z \langle \G_k^{nd}\rangle$ so that $\alpha \sim_\rat 0$. Then we have $\alpha \in \mathcal F^{k-1}$. 
\end{lemma}

\begin{lemma}\label{lem:crucial2}
We have $\mathbb Z \langle \G_{k}^{nd}\rangle  \cap \mathcal F^{k-1} = \mathcal F^k = \mathscr I_{k}^{nd}$. 
\end{lemma}

The rest of this section is devoted to the proof of these two lemmas. 

\medskip

Before giving the proof of Lemma~\ref{lem:crucial}, we introduce few extra notations, and state a useful proposition. 

\medskip

Let $1\leq l \leq k-1$ be an integer. For any $l$-simplex $\tau \in \G_l^{nd}$ with vertices $\u_0<\dots <\u_l$, fix a subset $J_\tau$ of $I(\tau)$ of size $l$ with the property that $\bigl| J_\tau \cap I(\u_j, \u_{j+1}) \,\bigr| =1$ for all $0\leq j\leq l-1$.

Define the projection $\pi_\tau: I(\tau) \to J_\tau$ which projects the subset $I(\u_j, \u_{j+1}) \subset I(\sigma)$ to the unique element of the intersection $I(\u_j, \u_{j+1}) \cap J_\tau$.

\medskip

For any set $S\in \mathcal A_{\tau}(k-l)$, define $\pi_\tau(S)$ as the multiset consisting of all the pairs $(\pi(i), \epsilon)$ for any pair $(i,\epsilon) \in S$. We have 
\begin{prop} \label{prop:keyprop} Let $1\leq l\leq k-1$ be an integer. For any $\tau \in \G_l^{nd}$ and any $S \in \mathcal A_\tau(k-l)$, we have 
\[C_\tau \prod_{(i,\epsilon) \in S} R^\epsilon_i - C_\tau \prod_{(i,\epsilon) \in \pi(S)} R^\epsilon_{i} \in \mathcal F^{l+1}.\]
\end{prop}
\begin{proof}
Let $m = k-l$, and denote by $(i_1,\epsilon_1), \dots, (i_{m}, \epsilon_m)$ all the elements of $S$.  Let $j$ be the index with 
$i_1 \in I(\u_j, \u_j+1)$. We can write 
\[R_{\tau, i_1}^{\epsilon_1} = R_{\tau, i_1}^{\epsilon_1} - R_{\tau, \pi(i_1)}^{\epsilon_1} + R_{\tau, \pi(i_1)}^{\epsilon_1} = (-1)^{\epsilon_1} R_{\tau, i_1,\pi(i_1)}+R_{\tau, \pi(i_1)}^{\epsilon_1}.\] By Proposition~\ref{prop:triv}, we have  
\[C_\tau R_{\tau, i_1,\pi(i_1)} = C_\tau \widetilde R_{\tau, i_1, \pi(i_1)} + \beta,\] 
for some $\beta\in \mathscr I_1$. 
Since $C_\tau \widetilde R^{\epsilon_1}_{i_1,\pi(i_1)} \in \mathbb Z\langle \G^{nd}_{l+1}\rangle$, setting $S':= S\setminus \{(i_1,\epsilon_1)\}$,
 we infer that $C_\tau \widetilde R^{\epsilon_1}_{i_1,\pi(i_1)} \prod_{(i,\epsilon) \in S'} R^{\epsilon}_{\tau, i} \in \mathcal F^{l+1}$.  Let $S_1 = S \cup \{(\pi(i_1), \epsilon)\} \setminus\{(i_1, \epsilon_1)\}$. Thus, since $\pi(\pi(i_1)) = \pi(i_1)$, we get
\[C_\tau  \prod_{(i,\epsilon) \in S} R^{\epsilon}_{\tau, i} - C_\tau \prod_{(i,\epsilon) \in S_1} R^{\epsilon}_{\tau, i} =  \in \mathcal F^{l+1}. \]
Proceeding by induction and applying the above reasoning to the set $S_t = S_{t-1} \cup\{(\pi(i_{t}),\epsilon_t)\} \setminus \{(i_t, \epsilon_t)\}$ and $(i_t, \epsilon_t) \in S_{t-1}$, for $t\geq 2$, we infer that for each $t$, 
\[C_\tau  \prod_{(i,\epsilon) \in S} R^{\epsilon}_{\tau, i} - C_\tau \prod_{(i,\epsilon) \in S_t} R^{\epsilon}_{\tau, i}  \in \mathcal F^{l+1}.\]
For $t=k-l$, we have $S_{k-l} = \pi(S)$, and the proposition follows.
\end{proof}

The following proposition is a direct consequence of the definition of simplicial structure on $\G$.
\begin{prop}\label{prop:triv2}
Let $\tau \in \G_l^{nd}$. For any $i\in I(\tau)$, we have
\[C_\tau (\sum_{\v\in \G_0}C_\v) - C_\tau (R^0_{\tau, i}+ R^1_{\tau,i}) \in \mathscr I_1.\]
\end{prop}

With these preliminaries, we are ready to prove Lemma~\ref{lem:crucial}.

\begin{proof}[Proof of Lemma~\ref{lem:crucial}] Let $\alpha \in \mathbb Z\langle \G^{nd}_k\rangle$ be an element with $\alpha \in \mathcal F^0 = \Rat(\G)$.  Proceeding by induction, we will show that for any $1\leq l \leq k-1 $, we have $\alpha \in \mathcal F^l$; for $l=k-1$ we get the lemma.

For the base of our induction, we need to show that $\alpha \in \mathcal F^1$. 
Since $\alpha \sim_\rat 0$, by definition of the rational equivalence, there are elements $\beta \in \mathscr I_1$ and $\alpha_0 \in \Z(\G)$,
and for any $\sigma \in \G_1^{nd}$, and $i\in I(\sigma)$ and $\epsilon\in \{0,1\}$, there is an element  
$\alpha_{\sigma,i,\epsilon} \in \Z^{k-1}(\G)$ such that we have 

\begin{equation}\label{eq5bis}
\alpha = \beta + \alpha_0(\sum_{\w\in \G_0}C_\w)+ \sum_{\substack{\sigma \in \G_1^{nd} \\ 
(i,\epsilon) \in \mathcal A_\sigma(1)}} 
\alpha_{\sigma, i, \epsilon} C_{\sigma} R_{\sigma,i}^{\epsilon}.
\end{equation}

For any $\u \in \G_0$, we compare the the coefficient of $C_\u^{k+1}$ on both sides of Equation~\eqref{eq5bis}. On the left hand side, the coefficient is zero since $\alpha$ has support in the non-degenerate simplices. 
Thus, the coefficient of $C_\u^{k+1}$ in the right hand side of the equality must be zero. Since all the monomials in the last sum have at least two distinct variables among $C_\v$, we infer that the coefficient of 
$C_\u^{k}$ in $\alpha_0$ must be zero. Therefore, we can write $\alpha_0 = \beta_0 + \sum_{\sigma \in \G_{1}^{nd}} \gamma_\sigma C_\sigma$ for  $\beta_0 \in \mathscr I_1$, and $\gamma_\sigma \in Z^{k-1}(\G)$. For any $\sigma \in \G_1^{nd}$, picking an arbitrary $i_\sigma \in I(\sigma)$, 
 and using Proposition~\ref{prop:triv2}, we decompose
\[C_\sigma \sum_{\w\in \G_0}C_\w = \beta_ \sigma + C_\sigma R^0_{\sigma,i_\sigma} + C_\sigma R^1_{\sigma, i_\sigma},\] 
for an element $\beta_\sigma \in \mathscr I_1$. Therefore, we have 
\begin{align*}
\alpha_0 (\sum_{\w\in \G_0} C_\w)=  \beta_0' + \sum_{\substack{\sigma \in \G_{1}^{nd} \\ \epsilon \in \{0,1\}}} \gamma_\sigma C_\sigma R^\epsilon_{\sigma, i_\sigma}.
\end{align*}
for some $\beta'_0\in \mathscr I_1$. 

Thus, replacing $\beta$ of Equation~\eqref{eq5bis} with $\beta + \beta_0'$, and replacing $\alpha_{\sigma, i_\sigma, \epsilon}$ for each $\sigma \in \G_1^{nd}$ and $\epsilon \in\{0,1\}$ with $\alpha_{\sigma, i_\sigma, \epsilon} + \gamma_{\sigma, i_\sigma, \epsilon}$, we can ensure to have $\alpha_0 =0$, and get 
\begin{equation}
\alpha = \beta + \sum_{\substack{ \sigma \in \G_1^{nd} \\ 
(i,\epsilon) \in \mathcal A_\sigma(1)}} 
\alpha_{\sigma, i, \epsilon} C_\sigma R_{\sigma,i}^{\epsilon}.
\end{equation}

Let $\sigma \in \G_1^{nd}$ and $(i,\epsilon) \in \mathcal A_\sigma(1)$. 
Each monomial in $\alpha_{\sigma,i,\epsilon}C_{\sigma}$ is either in $\mathscr I_1$ or has at least two distinct vertices. 
Thus, applying Proposition~\ref{prop:rewr} to each of the monomial terms in $\alpha_{\sigma,i,\epsilon}C_\sigma$, for any $\sigma \in \G_1^{nd}$, and $ (i, \epsilon)\in \mathcal A_\sigma(1)$, we finally infer the existence of $\beta_2 \in \mathscr I_{1}$, and for each $1\leq t\leq k-1$, the existence of integers $a_{\tau, S}$ associated to $\tau \in \G_{t}^{nd}$ and $S \in \mathcal A_\tau(k-t)$, such that 
we can write
\[\alpha =\beta_2 +  \sum_{t=1}^{k-1} \sum_{\substack{\tau \in \G_{t}^{nd} \\ S \in \mathcal A_\tau(k-t)}} a_{\tau,S}\, C_\tau 
\prod_{(i,\epsilon) \in S} R_{\tau,i}^\epsilon.\]
This shows that $\alpha \in \mathcal F^1$. 

\medskip

Assume now that we have $\alpha \in \mathcal F^l$ for an integer $1\leq l<k-1$. We shall prove that $\alpha \in \mathcal F^{l+1}$. By definition of $\mathcal F^l$, there is an element $\beta_l \in \mathscr I_1,$ and for all $l\leq t\leq k-1$, there are integers $a^l_{\tau,S} \in \mathbb Z$ associated to $\tau \in \G_t^{nd}$ and $S \in \mathcal A_\tau (k-t)$,  so that we have

\begin{equation}\label{eq7bis}
\alpha =\beta_l +  \sum_{t=l}^{k-1} \sum_{\substack{\tau \in \G_{t}^{nd} \\ S \in \mathcal A_\tau(k-t)}} 
a^l_{\tau,S}\, C_\tau 
\prod_{(i,\epsilon) \in S} R_{\tau,i}^\epsilon.
\end{equation}
In order to prove $\alpha \in \mathcal F^{l+1}$, it will be enough to show that 
\[\sum_{\substack{\tau \in \G_{l}^{nd} \\ S \in \mathcal A_\tau(k-l)}} 
a^l_{\tau,S}\, C_\tau 
\prod_{(i,\epsilon) \in S} R_{\tau,i}^\epsilon \in \mathcal F^{l+1}. \] 
 We show the following stronger statement.
\begin{claim} \label{claim1} We have for any $\tau \in \G_{l}^{nd}$
\[ \sum_{\substack{S \in \mathcal A_\tau(k-l)}} 
a^l_{\tau,S}\, C_\tau 
\prod_{(i,\epsilon) \in S} R_{\tau,i}^\epsilon \in \mathcal F^{l+1}.\]
\end{claim}

Fix a non-degenerate $l$-simplex $\tau \in \G_l^{nd}$ with vertices $\u_0<\dots <\u_l$. 
Let $J_\tau\subset I(\tau)$ be the set of size $l$ which intersects each $I(\u_j, \u_{j+1})$ in a unique element. By Proposition~\ref{prop:keyprop}, we have 
\begin{equation}\label{eq7}
\sum_{S\in \mathcal A_\tau(k-l)} a^l_{\tau, S} C_\tau \prod_{(i,\epsilon) \in S} R^\epsilon_{\tau, i} -  \sum_{S\in \mathcal A_\tau(k-l)} a^l_{\tau, S} C_\tau \prod_{(j,\epsilon) \in \pi(S)} R^\epsilon_{\tau, j} \in \mathcal F^{l+1}.  
\end{equation}

Denote by $i_0, \dots , i_{l-1}$ all the elements of $J_\tau$ with $i_j \in I(\u_{j}, \u_{j+1})$ for $j=0, \dots, l-1$.

\medskip

Proposition~\ref{prop:triv2} implies that for any for any $0\leq j<l$, we have 
\begin{equation}\label{eq8}
C_\tau R^0_{\tau, i_j} + C_\tau R^1_{\tau, i_j} -  C_\tau R^0_{\tau, i_l} - C_\tau R^1_{\tau, i_l} \in \mathscr I_1.
\end{equation}

\medskip

Let $\mathcal B_\tau(k-l)$ be the set of all mutisets $S$ of size $k-l$ such that each element $(i,\epsilon)\in S$ belongs to the set of pairs $\{(i_0, 0), (i_1, 0) ,\dots, (i_{l-1},0), (i_{l-1},1)\}. $

\medskip

Combining \eqref{eq7} and \eqref{eq8}, we infer the existence of integers $b_{\tau,S}$ for any $S \in \mathcal B_\tau(k-l)$ so that we have 
\begin{equation}\label{eq9}
\sum_{S\in \mathcal A_\tau(k-l)} a^l_{\tau, S} \, C_\tau \prod_{(i,\epsilon) \in S} R^\epsilon_{\tau, i} -  \sum_{S\in \mathcal B_\tau(k-l)} b_{\tau, S} \, C_\tau \prod_{(i,\epsilon) \in S} R^\epsilon_{\tau, i} \in \mathcal F^{l+1}.  
\end{equation}

Thus, Claim~\ref{claim1} will be a consequence of the following statement. 
\begin{claim}\label{claim2}
For any $\tau \in \G_{l}^{nd}$ and any $S \in \mathcal B_{\tau}(k-l)$, we have $b_{\tau, S}=0$. 
\end{claim}

Each element in $\mathcal B_\tau(k-l)$ is given by an ordered sequence ${\mathbf n}=(n_0, \dots, n_{l-1}, n_{l})$ of multiplicities of $(i_0, 0), \dots, (i_{l-1},0), (i_{l-1},1)$, respectively, with the property that $n_s\geq 0$, and $n_0+\dots+n_{l} =  k-l$: denote by $S_{\mathbf n}$ the element of $\mathcal A_\tau(k-l)$ associated to the ordered sequence $\mathbf n$.

For $\tau \in \G_{l}^{nd}$ with vertices $\u_0<\dots <\u_l$,  consider in the sum on the right hand side of Equation~\eqref{eq7bis}, the sum of the monomials which are in the polynomial ring $\mathbb Z[C_{\u_0}, \dots, C_{\u_l}]$. This polynomial is precisely 
\[C_{\tau} \sum_{\substack{{\mathbf n} = (n_0, \dots, n_{l}) \in \mathbb Z_{\geq 0}^{l+1}\\ n_0+\dots+n_l =k-l}}  b_{\tau, S_{\bf n}} \,\bigl(C_{\u_0}\bigr)^{n_0} \bigl(C_{\u_0}+C_{\u_1}\bigr)^{n_1} \dots \bigl(C_{\u_0} + \dots +C_{\u_{l-1}}\bigr)^{n_{l-1}} \bigl(C_{\u_l}\bigr)^{n_l},\]
which must be thus vanishing since $\alpha$ is supported on non-degenerated simplices, and $k-l\geq 1$.  
Since we have an isomorphism of polynomials rings 
$$\mathbb Z[C_{\u_0}, \dots, C_{\u_{l}}] \simeq \mathbb Z[C_{\u_0}, C_{\u_0}+C_{\u_1},\dots, C_{\u_0}+ \dots + C_{\u_{l-1}}, C_{\u_{l}}],$$ it follows that the coefficients $b_{\tau, S_{\mathbf n}}$ are all zero, which proves Claim~\ref{claim2}, and finishes the proof of Lemma~\ref{lem:crucial}. 
\end{proof}

We now prove Lemma~\ref{lem:crucial2}, finishing the proof of Theorem~\ref{thm:st}
\begin{proof}[Proof of Lemma~\ref{lem:crucial2}] Let $\alpha \in \mathbb Z\langle \G_k^{nd}\rangle \cap \mathcal F^{k-1}$. By definition of the filtration, we can write $\alpha$ in the form
\begin{equation}\label{eq11}
\alpha = \beta + \sum_{\substack{\tau \in \G_{k-1} \\ (i,\epsilon) \in \mathcal A_\tau(1)}} a_{\tau, (i,\epsilon)} C_\tau R_{\tau, i}^\epsilon.\end{equation}
for some $\beta \in \mathscr I_1$, and integers $a_{\tau, (i,\epsilon)}$ associated to $\tau \in \G_{k-1}^{nd}$ and $(i,\epsilon) \in \mathcal A_\tau(1)$.

\medskip

Let $\tau \in \G_{k-1}^{nd}$ with vertex set $\u_0<\dots <\u_{k-1}$. For each $0\leq j \leq k-2$, define 
\[\rho_{\tau,j} = \sum_{i\in I(\u_j, \u_{j+1})} a_{\tau, (i,1)}, \]
and for each $0\leq j\leq k-2$, define 
\[\ell_{\tau,j} = \sum_{i\in I(\u_j, \u_{j+1})} a_{\tau, (i,0)}.\]
For an integer $0\leq s\leq k-1$, the coefficient of $C_\tau C_{\u_s}$ on the right hand side of Equation~\eqref{eq11} is $\sum_{j=0}^{s-1} \rho_{\tau,j} + \sum_{j=s}^{k-2}\ell_{\tau,j}$, which, since $\alpha\in \mathbb Z\langle \G_{k-1}^{nd}\rangle$, must be zero. We infer that 
\begin{equation}\label{eq:van}
\,0\leq s\leq k-1, \qquad \sum_{j=0}^{s-1} \rho_{\tau,j} + \sum_{j=s}^{k-2}\ell_{\tau,j} =0. 
\end{equation}
Subtracting these equations for the values of $s=j$ and $s=j+1$,  we get  in particular, 
\[\forall \,0\leq j\leq k-2, \quad  \rho_{\tau,j} = \ell_{\tau,j}. \]
Applying Proposition~\ref{prop:triv2} to any $i\in I(\tau)$, we infer the existence of $\beta' \in \mathscr I_1$ such that we have 
\begin{align*}
\alpha &=\beta' +\sum_{i \in I(\tau)} \bigl(a_{\tau, (i,0)} - a_{\tau, (i,1)}\bigr) C_\tau R_{\tau, i}^0 + \Bigl(\sum_{i\in I(\tau)} a_{\tau, (i,1)}\Bigr) \bigl(\sum_{\w \in \G_0} C_\w\bigr) \\
&= \beta' +\sum_{i \in I(\tau)} \bigl(a_{\tau, (i,0)} - a_{\tau, (i,1)}\bigr) C_\tau R_{\tau, i}^0 + \bigl(\sum_{j=0}^{k-2} \rho_{\tau,j}\bigr) \bigl(\sum_{\w \in \G_0} C_\w\bigr)\\
& =  \beta' +\sum_{i \in I(\tau)} \bigl(a_{\tau, (i,0)} - a_{\tau, (i,1)}\bigr) C_\tau R_{\tau, i}^0  \quad  \textrm{(by vanishing Equation~\eqref{eq:van} for $s=k-1$).}
\end{align*}

Let $J_\tau$ be the subset of $I(\tau)$ of size $k-1$ which intersects each interval $I(\u_t, \u_{t+1})$ in a single element.  By Proposition~\ref{prop:keyprop}, we can further write 
\begin{align*}
\alpha &=\beta' +\sum_{i \in I(\tau)} \bigl(a_{\tau, (i,0)} - a_{\tau, (i,1)}\bigr) C_\tau R_{\tau, i}^0 \\
&= \beta'' + \sum_{i \in I(\tau)} \bigl(a_{\tau, (i,0)} - a_{\tau, (i,1)}\bigr) C_\tau \widetilde{R}_{\tau, i,\pi(i)}+  \sum_{i\in I(\tau)}\bigl(a_{\tau, (i,0)} - a_{\tau, (i,1)}\bigr) R_{\tau, \pi(i)}^0,
\end{align*}
for an element $\beta'' \in \mathscr I_1$. 
For each $0\leq j\leq k-2$, let $i_j$ to be the unique element in the intersection $I(\tau) \cap J_\tau$, so that we have $\pi(I(\u_j, \u_{j+1})) =\{i_j\}$.  We further get
\begin{align*}
\alpha &= \beta'' + \sum_{i \in I(\tau)} \bigl(a_{\tau, (i,0)} - a_{\tau, (i,1)}\bigr) C_\tau \widetilde{R}_{\tau, i,\pi(i)}+  \sum_{j=0}^{k-2} \sum_{i\in I(\u_j,\u_{j+1})}\bigl(a_{\tau, (i,0)} - a_{\tau, (i,1)}\bigr) R_{\tau, i_j}^0 \\
&= \beta'' + \sum_{i \in I(\tau)} \bigl(a_{\tau, (i,0)} - a_{\tau, (i,1)}\bigr) C_\tau \widetilde{R}_{\tau, i,\pi(i)}+  \sum_{j=0}^{k-2} \bigl(\ell_{\tau,j} - \rho_{\tau,j}\bigr) R_{\tau, i_j}^0 \\
&= \beta'' + \sum_{i \in I(\tau)} \bigl(a_{\tau, (i,0)} - a_{\tau, (i,1)}\bigr) C_\tau \widetilde{R}_{\tau, i,\pi(i)} \qquad \qquad \textrm{ (by the equality $\rho_{\tau,j} =\ell_{\tau,j}$).}
\end{align*}
We finally conclude that $\beta'' =0$ and $\alpha \in \mathscr I_{k}^{nd}$, and the lemma follows.
\end{proof}

  \subsection{Proof of the localization theorem} \label{sec:loc}

  In this section, we prove the localization theorem~\ref{thm1}. We retain the terminology from the previous sections.  For $\e\in \E = E_1 \times \dots \times E_d$, let $\square_\e =e_1\times \dots \times e_d$.  
  Regarding each edge $e_i$ as a subgraph of $G$ with the induced total order from $G_i$ on its vertices, 
  and applying the functoriality to the inclusions $e_i \hookrightarrow G_i$, we get 
  a map $\iota_{\e}^*: \Ch(\G) \rightarrow \Ch(\square_\e) \simeq \Ch(\square^d)$ associated to the 
  inclusion map of simplicial sets
 $$\iota_{\e}:\square_{\e} \hookrightarrow \G.$$ 
 By definition, the map $\iota^*_\e$ is identity on the generators associated to the vertices of 
 $\square_\e$, and is zero otherwise. For an element $\alpha \in Z(\G)$, and for $\e \in \E$, we denote by 
$\alpha_{|\e} \in Z(\square_\e)$ the  restriction of $\alpha$ to the hypercube $\square_{\e}$, i.e., $\alpha_{|\e}= \iota_\e^*(\alpha)$.

We need the following proposition which follows directly from the definition.
 \begin{prop}\label{prop:triv3} For any collection of connected subgraphs $H_1, \dots, H_d$ of $G_1,\dots, G_d$, respectively, let $\mathscr H = H_1 \times \dots \times H_d$ with its induced simplicial structure. We have for any $1\leq k \leq d$, $\mathscr I_{k}^{nd}(\mathscr H) \subseteq \mathscr I_{k}^{nd}(\G) .$
 \end{prop}

 With these notations, we first prove the injectivity part. 
 \begin{thm}\label{thm1-injectivity} The map of graded rings
$\Ch(\G) \rightarrow \prod_{\e \in \E} \Ch(\square_{\mathbf e})$ is injective. 
 \end{thm}

\begin{proof}
Let $\beta \in \Ch^k(\G)$ be an element such that 
 $\beta_{|\e} = 0$ in $\Ch^k(\square_\e)$ for all $\e\in \E$. We show that $\beta = 0$  in $\Ch(\G)$. By Theorem~\ref{thm:nd}, $\beta$ is represented in $\Ch^k(\G)$ by an element  $\alpha \in \mathbb Z\langle \G_{k}^{nd}\rangle$. Consider the element $\alpha$ such that the number of hypercubes $\e \in \E$ with $\alpha_{|\e} =0$ is maximized. We claim that $\alpha =0$ which proves the theorem.
 
 Suppose this is not the case, and consider a hypercube $\square_\e$ with $\alpha_{|\e} \neq 0$.  Since $\alpha_{|\e}$
 is zero in $\Ch(\square_\e)$, by Theorem~\ref{thm:st}, we get $\alpha_{|\e} \in \mathscr I_{k+1}^{nd}(\square_\e)$. Using the inclusion $\mathscr I_{k+1}^{nd}(\square_\e) \subseteq \mathscr I_{k+1}^{nd}(\G)$, which follows from Proposition~\ref{prop:triv3}, we find that $\alpha_{|\e}\sim_\rat 0$ in $\G$. Setting $\alpha'= \alpha -\alpha_{|\e}$, it follows that 
 $\alpha \sim_\rat \alpha'$ in $\G$.  On the other hand, we have $\alpha'_{|\e} =0$, and for any other $\e'$ with $\alpha_{|\e'} =0$, we also have $\alpha'_{|\e'}=0$. This contradicts the choice of $\alpha,$ and finishes the proof of the theorem.
 \end{proof}
 
 We now move to the proof of the second part of the theorem. 
 
 Let $(\alpha_\e)$ be a collection of elements in $Z^{k+1}(\square_\e)$, for 
$\e\in \E$, such that $(\alpha_\e)$ is in the kernel of the map $j$. It follows that for two hypercubes $\e$ and $\e'$ sharing a facet, we have 
$\alpha_{\e | \e'\cap \e} \sim_{\mathrm{rat}} \alpha_{\e'|\e'\cap \e}$ in 
$Z(\square_{\e\cap \e'})$. Using Theorem~\ref{thm:nd}, and Proposition~\ref{prop:func}, we can assume that $\alpha_\e \in \mathbb Z\langle \square_{\e,k}^{nd}\rangle$ for all $\e\in \E$. We show the existence of $\gamma_e \in \mathbb Z\langle \square_{\e,k}^{nd}\rangle$ for any $\e \in \E$ such that 
\begin{itemize}
\item[$(i)$]for any $\e\in \E$, we have $\gamma_\e \sim_\rat \alpha_\e$ in $\square_\e$; and 
\item[$(ii)$] for any two hypercubes $\e$ and $\e'$ sharing a facet, we have $\gamma_{\e|\e \cap \e'} = \gamma_{\e'|\e\cap \e'}$.
\end{itemize}
Assuming this, we get an element $\gamma \in \mathbb Z\langle \G_k^{nd}\rangle$ such that the class of $(\alpha_\e)_{\e\in \E}$ in $\prod_\e \Ch(\square_\e)$ is in the image of the restriction map $\Ch(\G) \rightarrow \prod_{\e\in \E}  \Ch(\square_\e)$ and the theorem follows. 

\medskip

Let $N =\bigl|\E\bigr|$, and enumerate all the elements of $\E$ as $\e_1, \dots, \e_N$.  Proceeding inductively, define $\gamma_{\e_1} = \alpha_{\e_1}$. For each $1\leq l \leq N-1$, suppose that $\gamma_{\e_1}, \dots, \gamma_{\e_l}$ have been defined, and $(i)$ and $(ii)$ are verified for $\gamma_{\e_1}, \dots, \gamma_{\e_l}$. Consider the hypercube $\e_{l+1}$. 
Denote by $\delta_1$ the restriction of $\alpha_{\e_{l+1}} - \gamma_{\e_1}$ to the intersection $\e_{l+1}\cap \e_1$.
Then $\delta_1 $ is rationally equivalent to zero, and so by Theorem~\ref{thm:st} and Proposition~\ref{prop:triv3}, belongs to $ \mathscr I_{k}^{nd}(\square_{\e_{l+1}\cap \e_1}) \subset \mathscr I_{k}^{nd}(\G)$.  
Set $\lambda_1 = \alpha_{\e_{l+1}} - \delta_1 \sim \alpha_{\e_{l+1}}$, and note that $\lambda_1$ and $\alpha_{\e_{1}}$ coincide on the intersection $\e_1\cap \e_{l+1}$. Proceeding, inductively, for $t=1, \dots, l$, define $\delta_t$ as the restriction of $\lambda_{t-1} - \gamma_{t}$ to the intersection $\e_{l+1}\cap \e_t$, and note that $\delta_t \in \mathscr I_{k}^{nd}(\square_{\e_{l+1}\cap \e_t}) \subset \mathscr I_{k}^{nd}(\G)$. Define $\lambda_{t+1} := \lambda_t - \delta_t$.

Defining $\gamma_{\e_{l+1}} := \lambda_{l+1}$, we get a collection of elements $\gamma_{\e_1}, \dots, \gamma_{\e_{l+1}}$ which verify $(i)$ and $(ii)$ above, and this completes the proof of the second part of  Theorem~\ref{thm1}.
 
\section{Combinatorics of the degree map: Proof of Theorem~\ref{thm:main1-intro}}\label{sec:hypercube}
Let $\square^d = \{0,1\}^d$ be the $d$-dimensional hypercube with its standard simplicial structure,  which is the $d$-fold product of the complete graph $K_2$ on two vertices $0<1$.
 After stating some results concerning the structure of the Chow ring $\Ch(\square^d)$, we prove Theorem~\ref{thm:main1-intro}.

First recall that the elements of $\square^d = \{0,1\}^d$ are the vertices of the hypercube $[0,1]^d$ in $\R^d$, and the non-degenerate $d$-simplices $\sigma$ of $\square^d$ are in bijection with the elements $\gamma$ of the permutation group $\mathfrak S_d$, as follows: denoting by $\mathbf e_1, \dots,\mathbf e_d $ the standard basis of $\R^d$, the $d$-simplex $\sigma_\rho$ associated to $\rho \in \mathfrak S_d$ has vertices ${\bf 0}, {\e}_{\rho(1)},  \e_{\rho(1)} + \e_{\rho(2)}, \dots,  \e_{\rho(1)} + \dots + \e_{\rho(d)}$. 
 We first state the following corollary of the structure Theorem~\ref{thm:st}.
\begin{prop} For any two non-degenerate $d$-simplices $\sigma_1, \sigma_2 \in \square^d$, we have $C_{\sigma_1} = C_{\sigma_2}$ in $\Ch(\square^d)$, and the Chow group $\Ch^{d+1}(\square^d)$ is canonically isomorphic to $\mathbb Z$.
\end{prop}
\begin{proof}
Since by Theorem~\ref{thm:nd}, the non-degenerate simplices generate the Chow ring,  the second part of the theorem follows from the first part. 
 
 So let $\sigma_1$ and $\sigma_2$ be two non-degenerate $d$ simplices of $\square^d$, and denote by $\rho_1, \rho_2$ the corresponding elements of $\mathfrak S_d$, respectively. 
 
  Writing $\rho_1^{-1}\rho_2$ as the product of the transpositions of the form $(i\,,\, i+1)$, for $1\leq i\leq d-1$,  it will be enough to prove the equality of $C_{\sigma_1}$ and $C_{\sigma_2}$ for $\rho_2 = \rho_1 (i,i+1)$. Furthermore, using the action of $\mathfrak S_d$ on $\square^d$ via permutation of the factors and Proposition~\ref{prop:permut}, we can further reduce to prove the equality of $C_{\sigma_1}$ and $C_{\sigma_2}$ for $\rho_1 = \mathrm{id}$ and  $\rho_2 = (i,i+1)$.

The vertices of $s_i$, for $i=1,2$, are $\bf 0, \e_{\sigma_i(1)}, \mathrm \e_{\sigma_i(1)}+ \e_{\sigma_i(2)}, \dots, \e_{\sigma_i(1)}+ \dots + \e_{\sigma_i(d)}$

 In this case, the vertices of $\sigma_1$ are $\v_0 =\bf 0$, and $\v_j=\e_1+\dots+ \e_j$, for $j\in[d]$, and the vertices of $\sigma_2$ are $\u_j$, $0\leq j\leq d$ with $\u_j = \v_j$ for $j\neq i$, and 
 $\u_{i}=\e_1+\dots+\e_{i-1}+\e_{i+1}$. Let $\tau$ be the $d-2$-simplex of $\square^d$ with vertices $\v_j=\u_j$ for $0\leq j\leq d$ and $j\neq i,i+1$. The vanishing of $C_{\sigma_1} - C_{\sigma_2}$ in $\Ch(\square^d)$ now follows by observing that 
 \[C_{\sigma_1} - C_{\sigma_2} = C_\tau \widetilde R_{\tau, i, i+1} \in \mathscr I_\rat.\]
\end{proof}

We define the degree  map 
\[\deg: \Ch^{d+1}(\square^d) \rightarrow \Z,\]
the canonical isomorphism of the previous proposition. 

\begin{cor} For any collection $G_1=(V_1, E_1), \dots, (G_d, E_d)$ of $d$ simple connected graphs, we have
$\Ch^{d+1}(\mathscr G) \simeq \mathbb Z^{|\mathscr E|}$.
\end{cor}
\begin{proof} This follows from the localization Theorem~\ref{thm:st} for $\Ch^{d+1}(\G)$, the previous proposition, and the vanishing of the Chow group $\Ch^{d+1}(\square^{d-1})$. 
\end{proof}

\begin{defi}[Degree map] 
For $\G =G_1 \times \dots G_d$ a $d$-fold product of simple connected graphs $G_1, \dots, G_d$, we define the degree map $\deg: \Ch^{d+1}(\G) \rightarrow \Z$ by 
$$\deg(x) := \sum_{\e\in \E}\, \deg_\e (\iota_\e^*(x)),$$
for any $x\in \Ch^{d+1}(\G)$, where $\deg_\e$ is the degree map of $\Ch^{d+1}(\square_\e) \simeq \Ch^{d+1}(\square^d)$.         
\end{defi}

\subsection{Intersection maps for the inclusion of hypercubes} 
Let $k<d$ be two natural numbers, and $\v\in \square^d$ be an element of length $|v|=k+1\leq d$.  Let  $I =\{i_1,\dots, i_{k+1}\}$  be support of $\v$, i.e., the subset of $[d]$ consisting of all the indices $i$ with $v_i =1$. The first $k$ indices $i_1, \dots, i_k$ define an inclusion 
$\eta: \square^k \hookrightarrow \square^d$  which is given by sending $\w = (w_1,\dots, w_k) \in \square^k$ to the point $\u=\eta(w)\in \square^d$ with
$u_{i_j}= w_j$ for all $j=1,\dots, k$, and $u_i=0$ for all $i\notin \{i_1,\dots, i_k\}$. 

\medskip

Denote as before by $K_2$ be the complete graph on two vertices $0<1$, and let $K_1$ be the complete graph on a unique vertex $0$. Note that $K_1 = K_2[\leq 0] = K_2[<1]$ in the terminology of Section~\ref{sec:generalities}. We can view the hypercube $\square^k$ as the product of graphs $H_1, \dots, H_d$, with $H_i =K_2$ for $i=i_1, \dots, i_k$, and $H_i = K_1$ for all the other values of $i$. The cube $\square^d$ corresponds to the $d$-fold product of the complete graph $K_2$. In this way the map $\eta: \square^k \to \square^d$ corresponds to the inclusion map  $H_1\times \dots \times H_d \hookrightarrow K_2\times \dots\times  K_2$. Consider now the map of $\Z$-modules $\beta=\beta_\v: \Z[\square^k] \rightarrow \Z[\square^d]$ defined by multiplication by $C_\v$
\[\forall\,i \in \mathbb N\,\,\,\forall\, \w_1, \dots, \w_i\in \square^k,\qquad \beta\bigl(C_{\w_1} C_{\w_2} \dots C_{\w_i}\bigr) := C_{\eta(\w_1)} C_{\eta(\w_2)}\dots C_{\eta(\w_i)}C_\v.\]
As a special case of Proposition~\ref{prop:intersection}, we get the following useful proposition. 
\begin{prop}\label{lem:embedding}
 The map $\beta$ induces a well-defined map of $\Z$-modules $\beta: \Ch(\square^k) \rightarrow \Ch(\square^d)$.
\end{prop}

With these preliminaries, we are ready to present the proof of Theorem~\ref{thm:main1-intro}.

  \subsection{Proof of Theorem~\ref{thm:main1-intro}}\label{sec:expdegree} Let $\sigma$ be any $d$-simplex of $\square^d$ with vertices $\v_1<\v_2<\dots<\v_k$ where each vertex $\v_i$ has has multiplicity $n_i$ in $\sigma$ for numbers $n_i\geq 1$ with $\sum_i n_i =d+1$, as in the theorem. To simplify the notation, let $C_i = C_{\v_i}$, and set $\alpha = C_1^{n_1}\dots C_k^{n_k}$.

We first prove the vanishing result, namely part (1) of the  theorem.  

\begin{claim}\label{claim:deg} Assume there exists an $1\leq i<k$ with $n_1+\dots+n_i > |\v_{i+1}|$. Then $\alpha =0$.
\end{claim}
\begin{claim}\label{claim:deg2} Assume there exists an $1< i\leq k$ with $n_i+\dots+n_k > d-|\v_{i-1}|$. Then $\alpha =0$.
\end{claim}
We only prove Claim~\ref{claim:deg}, as the proof of Claim~\ref{claim:deg2} is similar, and follows by symmetry. 

\begin{proof}[Proof of Claim~\ref{claim:deg}]  
We proceed  by a decreasing induction on $n_1+\dots+n_i + |\v_{i+1}|$. The base of our induction is 
the case $i=1$, $|\v_2| =1$, and $n_1 =2$. Note that since, $\v_1<\v_2$, this means $\v_1 = {\bf 0}$. It will be enough to prove that $C_1^2 C_2 =0$. Without loss of generality, and using Proposition~\ref{prop:permut}, we can suppose that $\v_2= \e_1$. By relation~$(\mathscr R3)$ in the Chow ring, we have
$$C_1 C_2 \,\sum_{\substack{\v \in \square^d\\ v_1 =0}}C_\v  =0.$$
Since for all $\v\neq \mathbf 0$ with $v_1=0$, by $(\mathscr R1)$, we have $C_\v C_1 C_2 =0$, we infer that $C_1^2 C_2 =0$.

\medskip

Let $N\geq 3$ be an integer, and suppose that the vanishing $\alpha =C_\sigma =0$ holds for any $\sigma$ verifying condition of the claim  for an $1\leq i<k$ such that $n_1+\dots+n_i+|\v_{i+1}|<N$.
We show the vanishing $\alpha=C_\sigma=0$ holds for any $\sigma$ verifying hypothesis of the claim for an $1\leq i < k$ with $n_1+\dots+n_i + |\v_{i+1}|=N$.
 
 \medskip

The proof is divided into the following three cases, depending on whether $i=1$, or $i\geq 2$ and $n_i\geq 2$, or $i\geq 2$ and $n_i=1$.
\medskip

\noindent $\bullet$ \emph{Consider first the case $i=1$}.  Thus, we have $n_1 > |\v_2|$. For any $j\in I(\v_1, \v_2)$, we have $v_{1,j}=0, v_{2,j}=1$, and  by relation $(\mathscr R3)$ in the Chow ring, we have  
 \begin{equation} \label{eq3}
 \sum_{\substack{\v\in \square^d\\v_j=0}} C_1^{n_1-1}C_\v C_{2}^{n_2} \dots C_{k}^{n_k}  =0, \textrm{ which in turn implies }
 \end{equation}
  \begin{equation}
\label{eq4} C_1^{n_1}\dots C_k^{n_k} + \sum_{\substack{\v \in \square^d\\ \v_1< \v< \v_2\,,\,v_j=0}} C_1^{n_1-1}C_\v C_{2}^{n_2} \dots C_{k}^{n_k} + \sum_{\substack{\v \in \square^d\\ \v< \v_1\,,\,v_j=0}} C_1^{n_1-1}C_\v C_{2}^{n_2} \dots C_{k}^{n_k}= 0.
\end{equation} 

\noindent $-$ For $\v \in \square ^d$ with $\v_1<\v <\v_2$, we have the vanishing of the product $C_{1}^{n_1-1}C_\v C_{2}^{n_2}\dots C_{k}^{n_k}=0$ by the hypothesis of our induction. Indeed, in this case, we have 
 $n_1-1 > |\v_2|-1\geq |\v|$ and $n_1-1+|\v_2|<N$. 
 
 \noindent $-$ For $\v \in \square ^d$ with $\v<\v_1$, we again have $C_\v C_{1}^{n_1-1} C_2^{n_2}\dots C_k^{n_k}=0$ by the hypothesis of our induction, since we have $1+n_1-1 = n_1>|\v_2| > |v_1|$ and $n_1-1+|\v_2|<N$. 
 
 The only remaining term in  Equation~\eqref{eq4} is $C_1^{n_1}\dots C_k^{n_k}$ which must be thus zero. 
 
 \medskip

\noindent $\bullet$ \emph{Consider now the case $i\geq 2$ and $n_i \geq 2$}. The proof in this case is similar to the above situation. Namely, take an index $j \in I(\v_i, \v_{i+1})$ so that $v_{i,j}=0$ and $v_{i+1,j}=1$. Using the equation of type $(\mathscr R3)$, 
 \[C_1^{n_1} \dots C_{i-1}^{n_{i-1}}C_i^{n_i-1}\Bigl(\sum_{\v\in \square^d: v_j =0}C_\v\Bigr)C_{i+1}^{n_{i+1}} \dots C_{k}^{n_k} =0,\]
 we see as above that all the terms $\v \neq \v_i$ with $v_j=0$ contribute zero to the above sum, i.e., 
 $C_1^{n_1} \dots C_{i-1}^{n_{i-1}}C_i^{n_i-1} C_\v C_{i+1}^{n_{i+1}} \dots C_{k}^{n_k} =0,$ either by the $(\mathscr R1)$ if the corresponding sequence does not form a simplex, or by the induction hypothesis as in the previous case. The only remaining term in the sum above is for $\v=\v_i$, and it follows that 
 $C_1^{n_1}\dots C_k^{n_k} =0$.

\medskip

\noindent $\bullet$ \emph{Finally, consider the case $i\geq 2$ and $n_i=1$}. In this case, we have $n_1+\dots+n_{i-1} = n_1+\dots+n_i -1 > |\v_{i+1}|-1\geq |\v_{i-1}|$ and by the hypothesis of our induction, we again have $C_1^{n_1}\dots C_k^{n_k} =0$.
\end{proof}

We now turn to the proof of the second part of the theorem. So suppose there is no  $1\leq i<k$ with $n_1+\dots+n_i > |\v_{i+1}|$, and there is no $2 \leq i\leq k$ with $n_i+\dots+n_k > d - |\v_{i-1}|$. Since $\sum n_i = d+1$, this means that for all $1\leq i \leq  k-1$, we have 
\[ |\v_i| +1 \leq n_1 +\dots +n_i \leq |\v_{i+1}|,\]
and, obviously, $|\v_k|+1\leq n_1+\dots+n_k=d+1$. We first show the existence of the sequence $x_i,y_i$ verifying the properties stated in the theorem.
 Let $y_0=|v_1|$, and define $x_i,y_i$, for $i=1, \dots, k$, as follows:
\[x_i := n_1+\dots+n_i - |v_i|-1, \quad \textrm{and} \quad y_i := |v_{i+1}| - n_1-\dots -n_i.\]
Note that $x_i,y_i \geq 0$ for all $i$, and $|\v_i| = |\v_{i-1}| + x_i + y_i+1$ for $2\leq i\leq k$, and $n_i = y_{i-1}+x_i+1$ for $1\leq i\leq k$. Thus $x_i, y_i$ verify the three conditions stated in part (2) of the theorem. 

\begin{claim}\label{claim:degree3} With the above notations, we have 
\[\deg(\alpha) = (-1)^{d+1-k}\binom{y_0+x_1}{y_0} \binom{x_1+y_1}{x_1}\binom{y_1+x_2}{y_1} \dots \binom{x_{k-1}+y_{k-1}}{y_{k-1}}\binom{y_{k-1}+x_k}{x_k}.\]
\end{claim}

We need to prove some preliminary results. 
\begin{prop} \label{prop:power1}We have $C_{\bf 1}^{d+1}= C_{\bf 0}^{d+1} = (-1)^d$.
\end{prop}
\begin{proof} By symmetry,  we will only need to show $C_{\bf 1}^{d+1}=(-1)^d$. Proceeding by induction, we show that for any $0\leq i\leq d-1$, we have 
$$C_{\bf 1}^{d+1} = (-1)^{i+1}C_{\bf 0} C_{\e_1} C_{\e_1+\e_2} \dots C_{\e_1+ \dots+\e_i} C_{\bf 1}^{d-i},$$
which, for $i=d-1$, gives the equality $C_{\bf 1}^{d+1}= (-1)^d$, as required. (For $i=0$, this means $C_{\bf 1}^{d+1} =  - C_{\bf 0} C_{\bf 1}^{d}$.) 

\medskip

First note that by $(\mathscr R2)$, we have $\bigl(\sum_{\v \in \square^2} C_\v\bigr) C_{\bf 1}^d =0$, which implies that 
\[C_{\bf 1}^{d+1} = - C_{\bf 0}C_{\bf 1}^{d} - \sum_{\v\neq \bf 0,\bf 1}C_\v C_{\bf 1}^d.\]
By vanishing part of Theorem~\ref{thm:main1-intro}, that we established in Claims~\ref{claim:deg} and~\ref{claim:deg2},  we have $C_\v C_{\bf 1}^d=0$ for all $\v\neq \bf 0,\bf 1$. This gives 
$$C_{\bf 1}^{d+1} = -C_{\bf 0}C_{\bf 1}^{d}.$$
Suppose that we have already proved for an $0\leq i<d-1$, that 
$$C_{\bf 1}^{d+1} = (-1)^{i+1}C_{\bf 0} C_{\e_1} C_{\e_1+\e_2} \dots C_{\e_1+ \dots+\e_i} C_{\bf 1}^{d-i}.$$
By relation $(\mathscr R3)$ in the Chow ring, we have 
\begin{equation}\label{eq5}
C_{\bf 0}C_{\e_1} C_{\e_1+\e_2} \dots C_{\e_1+ \dots+\e_i} \bigl(\sum_{\v\in \square^d, v_{i+1}=1}C_\v\bigr) C_{\bf 1}^{d-i-1} =0.
\end{equation}
 For any $\v \in \square^d$ with $v_{i+1}=1$, if $\e_1+ \dots +\e_i \not < \v$, by the definition of the simplicial structure of $\square^d$ and the relation $(\mathscr R1)$, we get $C_{\bf 0}C_{\e_1} C_{\e_1+\e_2} \dots C_{\e_1+ \dots+\e_i} C_\v C_{\bf 1}^{d-i-1} = 0$. 
  On the other hand, for any $ \e_1 + \dots +\e_{i}+ \e_{i+1} < \v < \bf 1$ with $v_{i+1}=1$, 
by applying the vanishing criterium of Theorem~\ref{thm:main1-intro},  we get 
$C_{\bf 0}C_{\e_1} C_{\e_1+\e_2} \dots C_{\e_1+ \dots+\e_i} C_\v C_{\bf 1}^{d-i-1} = 0$. Indeed, in this situation, we have $|\v|\geq  i+2$, which gives the inequality $d-i-1 > d-|\v|$ as required in Claim~\ref{claim:deg2}. It follows from these observations and Equation~\eqref{eq5} that 
$$C_{\bf 0} C_{\e_1} C_{\e_1+\e_2} \dots C_{\e_1+ \dots+\e_i} C_{\bf 1}^{d-i} = - C_{\bf 0} C_{\e_1} C_{\e_1+\e_2} \dots C_{\e_1+ \dots+\e_{i+1}} C_{\bf 1}^{d-i-1},$$
and the lemma follows.
\end{proof}

Let now $\v\in \square^d$ be any element of length $1<\ell = |\v|<d$ with support the subset $I \subset [d]$. Let $i_1\in I$ and $i_2\not \in I$ be two elements of $[d]$, and define $I_1 := I \setminus\{i_1\}$ and $I_2 := I^c \setminus \{i_2\}$.  

Let $m \in \mathbb N$ be an integer, and suppose $\w_1, \dots, \w_m$ are vertices of  $\square^d$ with support in $I_1$. Consider an element $\alpha_1 = C_{\w_1}^{a_1}\dots C_{\w_m}^{a_m}$ for $a_i \in \mathbb N$ with $a_1+\dots+a_m =\ell$. 

Similarly, let $t\in \mathbb N$ be an integer, and suppose ${\bf z}_1, \dots, {\bf z}_t \in \square ^d$ are such that for each $1\leq j \leq t $, $I\cup\{i_2\} \subseteq \mathrm{support}({\bf z}_j)$. Consider an element of the form $\alpha_2 = C_{{\bf z}_1}^{b_1}\dots C_{{\bf z}_t}^{b_t}$ for $b_j \in \mathbb N$ with $b_1 + \dots + b_t = d-\ell$.

 We write $\alpha_1|_{I_1}$ for the element  of $\Ch(\square^{\ell-1})$ of degree $\ell$ obtained by viewing $\w_1, \dots, \w_m$ in $\square^{\ell-1}$ (keeping only the coordinates in $I_1$) and keeping the exponents $a_1, \dots, a_m$. Similarly, we write ${\alpha_2}|_{I_2}$ for the element  of $\Ch(\square^{d-\ell-1})$ of degree $d-\ell$ obtained by restricting ${\bf z}_j$ to $I_2$, and keeping the exponents $b_1, \dots, b_t$.

\medskip

Then we have 
\begin{prop}\label{prop:aux1}
 Notations as above, we have
 \[\deg(\alpha_1 C_\v \alpha_2) = \deg(\alpha_1|_{I_1}) \deg(\alpha_2|_{I_2}).\]
\end{prop}
\begin{proof}
 Choose $\v_0 = {\bf 0} < \dots <\v_{\ell-1}<\v$ with support of $\v_i$ included in $I_1$ for each $1\leq i\leq \ell-1$. It follows from Proposition~\ref{lem:embedding} and the fact that $\Ch(\square^{\ell-1})$ is one dimensional in degree $\ell$ generated by $C_{\v_0} \dots C_{\v_{\ell-1}} C_\v$ that 
\[C_{\w_1}^{n_1}\dots C_{\w_m}^{n_m}C_\v = \deg(\alpha_1|_{I_1})C_{\v_0}C_{\v_1}\dots C_{\v_{\ell-1}}C_\v.\]
Similarly, chose $\v_{\ell+1}< \dots < \v_{d}$ such that $I\cup\{i_2\}$ is included in the support of $\v_{\ell+1}$. We have
\[C_\v C_{{\bf z}_1}^{b_1}\dots C_{{\bf z}_t}^{b_t} = \deg(\alpha_2|_{I_2})C_\v C_{\v_{\ell+1}}\dots C_{\v_{d}}.\]
We infer that 
\begin{align*}
 \alpha_1 C_\v \alpha_2 &=\deg(\alpha_1|_{I_1})C_{\v_0}\dots C_{\v_{\ell-1}}C_\v \alpha_2 \\
 &=\deg(\alpha_1|_{I_1})\deg(\alpha_2|_{I_2})C_{\v_0}C_{\v_1}\dots C_{\v_{\ell-1}}C_\v C_{\v_{\ell+1}}\dots C_{\v_{d}},
\end{align*}
from which the result follows. 
\end{proof}

The previous proposition allows to prove the following generalization of Proposition~\ref{prop:power1}.
\begin{prop} \label{prop:power2}
  For any $\v\in \square^d$, we have  $C_\v^{d+1} = (-1)^d \binom{d}{|\v|}$.
\end{prop}
\begin{proof}
We proceed by induction on $d$. The case $d=1$ trivially holds.
 Suppose the statement holds for $d-1$. 
 Let now $\v \in \square^d$ be an element with $|\v|=\ell$.  
We can suppose that $\v\neq {\bf 0}, {\bf 1}$, since we already treated these cases. We have by $(\mathscr R2)$
\begin{equation}\label{eq6}
\bigl(\sum_{\w\in \square^d}C_\w\bigr)C_\v^{d} =0.
\end{equation}

For all  $\w\neq {\bf 0},{\bf 1}, \v$, we have by applying either $(\mathscr R1)$ or by the vanishing criterium of Theorem~\ref{thm:main1-intro}, that $C_\w C_\v^d =0$. Therefore, from Equation~\eqref{eq6} we get
\[C_\v^{d+1} = - C_{\bf 0} C_\v^{d} - C_\v^d C_{\bf 1}.\]
Let $i \in \{1,\dots, d\}$ be an index with $v_i =1$, and define $I := [d] \setminus \{i\}$. By Proposition~\ref{lem:embedding}, we have $\deg(C_{\bf 0} C_{\v}^d) = \deg\bigl((C_{\v}|_{I})^d\bigr)$, which applying the hypothesis of the induction, gives
\[\deg(C_{\bf 0} C_{\v}^d) = \deg\bigl((C_{\v}|_{I})^d\bigr) = (-1)^{d-1} \binom{d-1}{|\v|-1}.\]
Similarly, let $j$ be an index with $v_j =0$, and $J = [d]\setminus\{j\}$. We have
\[\deg(C_{\v}^dC_{\bf 1}) = \deg(C_{\v}|_{J}^d) = (-1)^{d-1} \binom{d-1}{|\v|}.\]
The result now follows from the standard binomial identity $\binom{d}{|v|} =\binom{d-1}{|v|}+\binom{d-1}{|v|-1}.$
\end{proof}

We are ready to prove Claim~\ref{claim:degree3}, and complete the proof of Theorem~\ref{thm:main1-intro}.
\begin{proof}[Proof of Claim~\ref{claim:degree3}]
The proof goes by induction on $d$. So suppose that the statement holds in all hypercubes $\square^{d'}$ for any positive integer $d'<d$. We show that it holds also in $\Ch(\square^d)$.

We already treated the case $k=1$ in Proposition~\ref{prop:power2}. So we can suppose $k\geq 2$. 

Suppose first that $n_1=1$. In this case we must have $\v_1 = {\bf 0}$, since otherwise, we would have $n_2+ \dots + n_k = d > d-|\v_1|$, and by Claim~\ref{claim:deg2}, we would have $\alpha=0$.
Let now $i\in [d]$ be an index with  $v_{2,i}=1$, and let $I = [d]\setminus\{i\}$ and $\alpha_2 = C_{\v_2}^{n_2} \dots C_{\v_k}^{n_k}$. By Proposition~\ref{lem:embedding}, we get
$\deg(C_{\bf 0} \alpha_2) = \deg(\alpha_2|_{I})$, and the statement then follows by the induction hypothesis in the hypercube $\square^{d-1}$.

\medskip

So we can suppose that $n_1 \geq 2$. We divide the proof into two parts depending on whether $\v_1={\bf 0}$ or not. 

\medskip

 Suppose first that $\v_1 \neq 0$. Let $i \in I(\v_1, \v_2)$, so we have $v_{1,i} =0$ and $v_{2,i}=1$. By relation $(\mathscr R3)$, we get
\[C_{\v_1}^{n_1-1}\Bigl[C_{\v_1}\Bigl(\sum_{\w\in \square^d, w_i=0} C_\w \Bigr) C_{\v_2}\Bigr] C_{\v_2}^{n_2-1}C_{\v_3}^{n_3} \dots C_{\v_k}^{n_k}=0,\]
in the Chow ring, from which we deduce, by developing, and using $(\mathscr R1)$ and the vanishing criterium in Theorem~\ref{thm:main1-intro}, that
\begin{align*}
 C_{\v_1}^{n_1}\dots C_{\v_k}^{n_k} = - C_{\bf 0} C_{\v_1}^{n_1-1}\dots C_{\v_k}^{n_k} - \sum_{\substack{\v_1<\v < \v_2\\
 |\v| = |\v_1|+x_1}} C_{\v_1}^{n_1-1}C_\v C_{\v_2}^{n_2}\dots C_{\v_k}^{n_k}.
\end{align*}
Therefore, we have 

\begin{equation}\label{eq7:4}
\deg(C_{\v_1}^{n_1}\dots C_{\v_k}^{n_k}) = - \deg(C_{\bf 0} C_{\v_1}^{n_1-1}\dots C_{\v_k}^{n_k}) - \sum_{\substack{\v_1<\v < \v_2\\
 |\v| = |\v_1|+x_1}} \deg(C_{\v_1}^{n_1-1}C_\v C_{\v_2}^{n_2}\dots C_{\v_k}^{n_k}).
 \end{equation}
 Let $j$ be an arbitrary element in the support of $\v_1$, and set $J :=[d] \setminus \{j\}$.  From Proposition~\ref{prop:aux1}, we get 
 \begin{align*}
 \deg&(C_{\bf 0} C_{\v_1}^{n_1-1}\dots C_{\v_k}^{n_k}) = \deg\Big(\bigl(C_{\v_1}^{n_1-1}\dots C_{\v_k}^{n_k}\bigr)|_J\Bigr) \\
 &= (-1)^{d-k}\binom{y_0+x_1-1}{y_0-1} \binom{x_1+y_1}{x_1}\binom{y_1+x_2}{y_1} \dots \binom{x_{k-1}+y_{k-1}}{y_{k-1}}\binom{y_{k-1}+x_k}{x_k}.
 \end{align*}
 In the last inequality we used the hypothesis of our induction in $\square^{d-1}$.
 
 \medskip
 
 Now for each $\v_1<\v<\v_2 $ with $|\v| = |\v_1| + x_1 = y_0+x_1 = n_1-1$, let $i_{\v,1} \in I(\v_1, \v)$ be an  arbitrary element, and define $I_{\v,1} = \mathrm{support}(\v) \setminus \{i_{\v,1}\}$. Similarly, let $i_{\v,2} \in I(\v, v_2)$ be an arbitrary element, and define $I_{\v,2} = [d] \setminus \Bigl(\mathrm{support}(\v)  \cup\{i_{\v,2}\}\Bigr)$.
By Proposition~\ref{prop:aux1}, we have 
 \begin{align*}
  \deg&(C_{\v_1}^{n_1-1}C_\v C_{\v_2}^{n_2}\dots C_{\v_k}^{n_k})  =  \deg(C_{\v_1}^{n_1-1}|_{I_{\v,1}}) \deg\Bigl(\bigl(C_{\v_2}^{n_2}\dots C_{\v_k}^{n_k}\bigr)|_{I_{\v,2}} \Bigr)\\
  &= (-1)^{n_1-2} \binom{n_1-2}{|v_1|} \times (-1)^{d-|v|-k+1}\binom{y_1+x_2}{y_1} \dots \binom{x_{k-1}+y_{k-1}}{y_{k-1}}\binom{y_{k-1}+x_k}{x_k}\\
  &= (-1)^{n_1-2 + d-|v|-k+1}\binom{y_0+x_1-1}{y_0} \binom{y_1+x_2}{y_1} \dots \binom{x_{k-1}+y_{k-1}}{y_{k-1}}\binom{y_{k-1}+x_k}{x_k}\\
  &= (-1)^{d-k}\binom{y_0+x_1-1}{y_0} \binom{y_1+x_2}{y_1} \dots \binom{x_{k-1}+y_{k-1}}{y_{k-1}}\binom{y_{k-1}+x_k}{x_k}
 \end{align*}
Since $|\v_2| = x_1+y_1$, there are in total $\binom{x_1+y_1}{x_1}$ choices for $\v_1<\v <\v_2$ with $|\v| = |\v_1|+x_1$. It finally follows from Equation~\eqref{eq7:4} and the calculation of degrees above that 
\begin{align*}
\deg(\alpha) &= (-1)^{d+1-k} \Bigl[\,\binom{y_0+x_1-1}{y_0-1} + \binom{y_0+x_1-1}{y_0}\Bigr]\binom{x_1+y_1}{x_1} \dots \binom{y_{k-1}+x_k}{x_k}\\
&=(-1)^{d+1-k} \,\binom{y_0+x_1}{y_0}\binom{x_1+y_1}{x_1} \dots \binom{y_{k-1}+x_k}{x_k},
\end{align*}
and the theorem follows. 

\medskip

In the final case $\v_1 = {\bf 0}$, using $(\mathscr R1)$ and the vanishing criterium in Theorem~\ref{thm:main1-intro}, we have, similarly as in the previous case above, that  
\begin{align*}
 C_{\v_1}^{n_1}\dots C_{\v_k}^{n_k} = - \sum_{\substack{\v_1<\v < \v_2\\
 |\v| = |\v_1|+x_1}} C_{\v_1}^{n_1-1}C_\v C_{\v_2}^{n_2}\dots C_{\v_k}^{n_k}.
\end{align*}
from which the result again follows by the hypothesis of our induction using a similar argument as in the previous case $\v_1 \neq {\bf 0}$. 
  \end{proof}

\section{Fourier transform and a dual description of $\Ch(\square^d)$}\label{sec:fourier}   Identify the points of $\square^d$ with the elements of the vector space $\mathbb F_2^d$, and consider the scalar product $\langle\,,\rangle$ on $\mathbb F_2^d$ defined by 
$\langle\v,\u\rangle  = \sum_{i=1}^d v_i.u_i \in \mathbb F_2,$ for any $\u,\v\in \mathbb F_2^d$. Recall that for any $\w \in \mathbb F_2^d$, we defined $F_\w$ by
\[F_\w := \sum_{\v\in\square^d} (-1)^{\langle \v,\w\rangle}C_\v,\]
and noticed that by Fourier duality,  the set $\{F_\w\}_{\w\in \mathbb F_2^d}$ forms another system of generators for the localized Chow ring $\Ch(\square^d)[\frac 12]$.  The following set of relations are verified by $F_\w$s in $\Ch(\square^d)$~\cite{Kolb1}.

\begin{itemize}
 \item[($\mathscr R^* 1$)] For any $\bf w \in \mathbb F^d$, we have $F_{\bf 0} F_{\bf w}=0$;
 \item[($\mathscr R^*2$)] For any $i\in [d]$, and any $\bf w, \bf z \in \mathbb F^d$, we have $F_{\e_i}(F_{\bf w}-F_{\w+\e_i})(F_{\bf z}+ F_{\bf z+\e_i})=0$;
 \item[($\mathscr R^*3$)] For any pair of indices $i,j\in[d]$, and any $\bf w, \bf z$, we have $(F_{{\bf w}+\e_i+\e_j}-F_{\bf w})(F_{{\bf z}+\e_i+\e_j}-F_{\bf z}) = (F_{{\bf w}+\e_i}-F_{{\bf w}+\e_j})(F_{{\bf z}+\e_i}-F_{{\bf z}+\e_j})$.
\end{itemize}
Consider the ideal $\widetilde \Rat(\square^d)$ of $\Z[F_{\bf w}]$ generated by ($\mathscr R^*1$), ($\mathscr R^*2$), and ($\mathscr R^*3$), and define $\widetilde \Ch(\square^d) := \Z[F_{\bf w}]/\widetilde \Rat(\square^d).$ 

We now give a proof of these relations in the Chow ring, proving at the same time Theorem~\ref{thm:iso}, which shows that $\widetilde \Ch(\square^d)[\frac 12]  =\Ch(\square^d)[\frac 12]$. 

\begin{proof}[Proof of Theorem~\ref{thm:iso}]
 We shall show that inverting $2$, we have
$\widetilde \Rat(\square^d) = \Rat(\square^d)$, and $\widetilde \Ch(\square^d)  =\Ch(\square^d)$. 

First note that $F_{\bf 0} = \sum_{\mathbf{v} \in \mathbb F_2^d} C_{\mathbf v},$
and so for any $\w$, we have
\[F_{\mathbf 0} F_{\mathbf w} = \sum_{\mathbf{u} \in \mathbb F_2^d} (-1)^{\langle \u, \w\rangle}\Bigl(\sum_{\v\in \mathbb F_2^d} C_{\bf v}\Bigr)C_{\bf u}.\]
This shows that $F_{\bf 0} F_{\bf w} \in \mathscr I_\rat$, and yields to the proof of ($\mathscr R^*1$). On the other hand, we see from the above description, and using the Fourier duality, that for any $\u \in \mathbb F_2^d$,
\[\Bigl(\sum_{\v\in \mathbb F_2^d} C_{\v}\Bigr)C_{\u} = \frac 1{2^d} \sum_{{\bf w} \in \mathbb F_2^d}(-1)^{\langle \u, \w\rangle} F_{\bf 0}F_{\bf w}. \] 
This shows that any generator $\Bigl(\sum_{\bf v} C_{\bf v}\Bigr)C_{\bf u}$ of type ($\mathscr R2$) in $\Rat(\square^d)$ 
 belongs to the ideal $\widetilde\Rat(\square^d)$ of $\Z[\frac 12][F_{\bf w}]$.

\medskip

Let $\e = \e_i$ for some $i\in [d]$. Note that for any ${\bf w}, {\bf z} \in \mathbb F_2^d$, we have
\begin{equation}\label{eq10}
 F_{\bf w} - F_{\bf w+\e} = 2\sum_{\substack{\u\in \mathbb F_2^d: \,\,u_i =1}} (-1)^{\langle \u,{\bf w}\rangle} C_{\u}\qquad , \qquad F_{\bf z} + F_{{\bf z}+\e} = 2\sum_{\substack{\v\in \mathbb F_2^d:\,\, v_i =0}}
(-1)^{\langle \v,{\bf z}\rangle} C_{\v}.
\end{equation}

For any $\epsilon \in \{0,1\}$, define 
\[R^\epsilon_{i} := \sum_{\substack{{\bf y} \in \mathbb F_2^d\\ y_i=\epsilon}}C_{\bf y}.\]
From Equation~\eqref{eq10}, we get
\begin{align}
 F_{\e}(F_{\bf w}-F_{\w+\bf e})(F_{\bf z}+ F_{\bf z+\e})&= 4 \Bigl(\sum_{{\bf y} \in\mathbb F_2^d} (-1)^{y_i}C_{\bf y}\Bigr)\Bigl(\sum_{\substack{\u\in \mathbb F_2^d \\ u_i =1}} (-1)^{\langle \u,{\bf w}\rangle} 
 C_{\u}\Bigr)\Bigl(\sum_{\substack{\v \in \mathbb F_2^d\\ v_i =0}}(-1)^{\langle\v,{\bf w}\rangle} C_{\v}\Bigr) \\
 \label{eq2}&=4\sum_{\substack{\epsilon\in \mathbb \{0,1\}; \u, \v \in \mathbb F_2^d\\u_i =1, v_i=0}} (-1)^{\epsilon+\langle\u,{\bf w}\rangle+\langle \v, {\bf z}\rangle})  C_{\u}C_{\v} R^\epsilon_{i},
\end{align}
Since $C_{\u}C_{\v} R^\epsilon_{i} \in \mathscr I_\rat$, this  implies the relation ($\mathscr R^*2$) among the $F_{\w}$ in the Chow ring $\Ch(\square^d)$. 

Define now the two functions $f, g : \mathbb F_2 \times \mathbb F_2^{d} \times \mathbb F_2^{d} \rightarrow \Z[C_w]$, as follows. For 
any triple $(\epsilon, \u, \v) \in \mathbb F_2 \times \mathbb F_2^{d} \times \mathbb F_2^{d}$, set
\[f(\epsilon, \u,\v) := \begin{cases}
                         C_{\u}C_\v R^\epsilon_i & \qquad \textrm{if $u_i =1$, $v_i=0$}\\
                         0 & \qquad \textrm{otherwise}.
                        \end{cases} \]
 For any triple $(\delta, {\bf w},{\bf z}) \in \mathbb F_2 \times \mathbb F_2^{d} \times \mathbb F_2^{d}$, set
\[g(\delta, {\bf w},{\bf z}) = \begin{cases}
                                 F_{\e}(F_{\bf w}-F_{\w+\e})(F_{\bf z}+ F_{{\bf z}+\e}) &\qquad \textrm{if $\delta=1$}\\
                                 \sum_{\substack{\u,\v \in \mathbb F_2^d \\ u_i=1, v_i=0}} (-1)^{\langle\u, {w} \rangle +\langle \v, {\bf z} \rangle} C_{\u}C_\v F_{\bf 0}& \qquad \textrm{if $\delta=0$}.                    
                            \end{cases}
                            \]
 Note that we have for any $(\delta, \w, {\bf z}) \in \mathbb F_2 \times \mathbb F_2^d \times \mathbb F_2^d$,
\[g(\delta, {\bf w},{\bf z}) = 4\sum_{\substack{\epsilon \in\{0,1\}\\ \u,\v \in \mathbb F_2^d}} (-1)^{\delta.\epsilon + \langle \u,{\bf w}\rangle+\langle \v,{\bf z}\rangle} {f}(\epsilon, \u,\v).\]
Indeed, for $\delta=1$ this is identical to Equation~\eqref{eq2}, and for $\delta=0$, both sides of the equations are equal by definition (the right hand side is a sum of the terms of the form 
$C_{\u}C_{\v} (R_i^1+R^0_i) =C_{\u}
C_\v(\sum_{\bf x}C_{\bf x}) = C_\u C_\v F_{\bf 0} $.   By Fourier duality in $\mathbb F_2^{2d+1}$, it follows that 
\[{f}(\epsilon, \u,\v) = \frac 1{2^{2d-1}}\sum_{(\delta, \w,{\bf z})\in \mathbb F_2^{2d+1}}(-1)^{\delta.\epsilon + \langle \u,{\bf w}\rangle+\langle \v,{\bf z}\rangle}g(\delta, {\bf w},{\bf z}).\]
Since $g$ takes values in $\widetilde{\Rat}(\square^d)$, this shows that for any $\epsilon, \u,\v$, we have ${f}(\epsilon, \u,\v)\in \widetilde \Rat(\square^d)$. In particular,  inverting $2$, 
all the generators of type ($\mathscr R3$) in $\Rat(\square^d)$ belong to $\widetilde \Rat(\square^d)$.

\medskip

Finally, to prove ($\mathscr R^*3$), let $i,j\in [d]$, and $\e=\e_i$ and $\e'=\e_j$. By an easy computation,  we have
\begin{align*}
 (F_{\bf w+\e+\e'}-F_{\bf w})(F_{\bf z+\e+\e'}-F_{\bf z})  -&(F_{\bf w+\e}-F_{\bf w+\e'})(F_{\bf z+\e}-F_{\bf z+\e'})\\
 &= \sum_{\u,\v \in \mathbb F_2^d} (-1)^{\langle \u, {\bf w}\rangle + \langle\v, {\bf z}\rangle } h(u_i,u_j,v_i,v_j) C_{\u}C_{\v},
\end{align*}
where $h : \mathbb F_2^4 \rightarrow \mathbb Z$ is the function given by
\[  h(a,b,c,d) :=\Bigl((-1)^{a+b} -1 \Bigr)\Bigl((-1)^{c+d} -1 \Bigr)-\Bigl((-1)^{a} -(-1)^{b} \Bigr)\Bigl((-1)^{c} -(-1)^{d} \Bigr),\]
for any $(a,b,c,d) \in \mathbb F_2^4.$ In particular, we get $h(0,1, 1,0) = h(1,0,0,1) =8$, and $h$ vanishes at all other points of $\mathbb F_2^4$. Now, note that for any pair of points $\u,\v \in \mathbb F_2^d$ such that $\{(u_i,u_j),(v_i,v_j)\} = \{(0,1), (1,0)\}$, since $\u$ and $\v$ do not form a simplex in $\square^d$, we have $C_\u C_\v \in \Rat(\square^d)$. 
This proves that the relation ($\mathscr R^*3$) holds among the $F_\w$. 

To conclude, using Fourier duality in 
$\mathbb F_2^d \times \mathbb F_2^d$, 
we see that for any pair $i,j \in [d]$, 
$2^{2d}h(u_i,u_j, v-i,v_j)C_\u C_\v$ is a linear combination with $\pm$ coefficient of the terms 
$(F_{{\bf w}+\e+\e'}-F_{\bf w})(F_{{\bf z}+\e+\e'}-F_{\bf z}) - 
  (F_{\bf w+\e}-F_{\bf w+\e'})(F_{\bf z+\e}-F_{\bf z+\e'})$, and thus belongs to
  $\widetilde \Rat(\square^d)$. For $\u$ and $\v$ which do not form a simplex
in $\square^d$, there are indices $i,j\in [d]$, 
such that $\{(u_i, u_j), (v_i,v_j)\} = \{(0,1), (1,0)\}$. Since $h(0,1,0,1) =8$, we infer that
$C_\u C_\v \in \widetilde \Rat(\square^d)$.

This shows all the relations of type ($\mathscr R1$) in $\Rat(\square^d)$ belong to $\widetilde\Rat(\square^d)$, 
which finally shows that 
$\widetilde \Rat(\square^d)=\Rat(\square^d)$ in $\Z[\frac 12][C_\v]=\Z[\frac 12][F_\w]$, and the theorem follows.
\end{proof}
  
\subsection{Functoriality of $\widetilde{\Ch}$ for the inclusion of hypercubes} 
Let $r<d$ be two integers.  Let $\widetilde \Ch(\square^d)$ be the chow ring of $\square^d$, with generators $F_\u$ for $\u \in \mathbb F_2^d$ and with relations $(\mathscr R^*1), (\mathscr R^*2), (\mathscr R^*3)$. Similarly, let $\widetilde \Ch(\square^r)$ be the chow ring of $\square^r$, with generators $\widetilde F_\w$ for $\w \in \mathbb F_2^r$ and with relations $(\mathscr R^*1), (\mathscr R^*2), (\mathscr R^*3)$. 

Let $I \subseteq \{1,\dots, d\}$ be a subset of size $r$. Viewing $\mathbb F_2^r$ as the set of elements of $\mathbb F_2^d$ with support in $I$, we get an inclusion $\eta: \square^r \hookrightarrow \square^d$. 

\begin{prop}\label{prop:func2}
We have a morphism of graded rings 
$\eta_*: \widetilde \Ch(\square^r) \rightarrow \widetilde \Ch(\square^d)$ defined by sending $\widetilde F_\w $ to $F_{\eta(\w)}$. 
\end{prop}  
\begin{proof} Consider the map of polynomial rings 
$\Z[F_{\w}]_{\w\in \mathbb F_2^r} \to \Z[F_{\u}]_{\u\in \mathbb F_2^d}$ induced by $\eta$. 
The image by $\eta$ of any relation of type $(\mathscr R^*1), (\mathscr R^*2), (\mathscr R^*3)$ in 
$\widetilde \Ch(\square^r)$ is a relation of type 
$(\mathscr R^*1), (\mathscr R^*2), (\mathscr R^*3)$ in $\widetilde \Ch(\square^d)$. 
It follows that passing to the quotient, we get a well-defined map $\eta_*: \widetilde \Ch(\square^r) \rightarrow \widetilde \Ch(\square^d)$ of Chow rings. 
\end{proof}
\begin{cor} The inclusion $\eta: \square^I \rightarrow \square^d$ induces a map of localized Chow rings 
\[\eta_*: \Ch(\square^r)[\frac 12] \rightarrow \Ch(\square^d)[\frac12].\]
\end{cor}
\begin{proof} This follows from the previous proposition and the isomorphism of Theorem~\ref{thm:iso}.
\end{proof}
The morphism $\eta: \square^r \rightarrow \square^{d}$ is induced from a morphism of graphs 
$\eta_i: H_i \rightarrow K_2$, for $i=1, \dots, r$, with $K_2$ the complete 
graph on two vertices $0<1$, and $H_i =G_i$ for $i\in I$, and $H_i = K_2[<1] \simeq K_1$ for 
$i\notin I$. It follows from Proposition~\ref{prop:func}, that we have a morphism of Chow rings 
$\eta^* : \Ch(\square^d)[\frac 12] \rightarrow \Ch(\square^d)[\frac 12]$. 
\begin{prop} The composition morphism $\eta^*\circ \eta_*$ of $\Ch(\square^r)[\frac 12]$ is the identity.   
\end{prop}
\begin{proof} It will be enough to prove this in degree one. Let $\w \in \square^r$. We have 
\begin{align*}
\eta^*\circ \eta_*(\widetilde F_\w) = \eta^*(F_{\eta(\w)}) &= \eta_*\bigl( \sum_{\v\in \mathbb F_2^d} (-1)^{\langle \v,\eta(\w) \rangle } C_\v \bigl) \\
&= \sum_{\v\in \mathbb F_2^d} (-1)^{\langle \v,\eta(\w) \rangle }  \eta_*(C_\v) = \sum_{{\bf z} \in \mathbb F_2^r} (-1)^{\langle {\bf z},\w\rangle} C_{\bf z}=\widetilde F_{\w}.
\end{align*}
\end{proof}
\subsection{Vanishing theorem} In this section, we prove the vanishing Theorem~\ref{thm:vanishing-intro}. 
So consider elements $\w_0,\dots, \w_d \in \mathbb F_2^d$.    Let $\mathcal P = \{P_1,\dots,P_k\}$ be 
a partition of $[d]$ into $k$ disjoint non-empty sets. Recall that for each $\w_i$, we denote by 
$\alpha(\w_i,\mathcal P)$ the number of indices $1\leq i\leq k$ such that there exists an index 
$j\in P_i$ with $\w_j=1$. Suppose that the condition in Theorem~\ref{thm:vanishing-intro} is verified, namely, $\sum_{i=0}^d \alpha(\w_i, \mathcal P) <d+k$. 
We have to prove that $F_{\w_0}\dots F_{\w_d}=0$.

\medskip

 We first reformulate the condition in the theorem as follows. 
 Let $\w_0,\dots,\w_d \in \mathbb F_2^d$, possibly with $\w_i =\w_j$ for $i \neq j$.  
 We construct a bipartite graph $H = H(\mathcal P; \w_0, \dots, \w_d)$ as follows.  The graph 
 $H$ has the vertex set partitioned into two separate parts $W$ and $V$ of size $d+1$ and $d$, respectively. Let 
 $W = \{\omega_0,\dots, \omega_d\}$ be $d+1$ vertices on one side, and 
$V = \{1,\dots,d\}$ the $d$ vertices on the other side. 
 There is an edge between $\omega_i \in W$ and $j\in V$ if the $j$-coordinate of $\w_i$ is one. We have the following proposition.

\begin{prop} The following three conditions are equivalent
\begin{itemize}
 \item[$(i)$] The graph $H$ is disconnected;
 \item[$(ii)$] There exists a partition $\mathcal P=\{P_1,P_2\}$ of $\{1,\dots, d\}$ such that $\sum_{i=0}^d \alpha(\w_i, \mathcal P) <d+2$;
 \item[$(iii)$] There exists an integer $k\in \mathbb N$ and a partition $\mathcal P=\{P_1,P_2, \dots, P_k\}$ such that 
 $$\sum_{i=0}^d \alpha(\w_i, \mathcal P) <d+k.$$
\end{itemize}
\end{prop}

\begin{proof}
 To show that $(i)$ implies $(ii)$, note that if $H$ is disconnected, 
 then we can write $W = W_1\sqcup W_2$ and $V=V_1\sqcup V_2$ such that all the edges of $H$ join a vertex of 
 $W_i$ to a vertex of $V_i$, for $i=1,2$. 
 In this case, let $\mathcal P = \{V_1,V_2\}$. Then we have 
 $\sum_{i=0}^d \alpha(\w_i, \mathcal P) \leq |W_1|+ |W_2| = d+1 < d+2$.
 
 Obviously, $(ii)$ implies $(iii)$. 
 
 Finally, let $\mathcal P=\{P_1,P_2, \dots, P_k\}$ be 
 a partition of $V$ such that $\sum_{i=0}^d \alpha(\w_i, \mathcal P) <d+k$. Let $\widetilde H$ 
 be the graph obtained from $H$ by contracting each $P_i$ 
 into a single vertex.
 The number of edges of 
 $\widetilde H$ is precisely $\sum_{i=0}^d \alpha(\w_i, \mathcal P)$. On the other hand, the
 number of vertices of $\widetilde H$ is $d+1+k$. Since $|E(\widetilde H)| < |V(\widetilde H)|-1$, 
 the graph $\widetilde H$ 
 is not connected, which shows that $H$ cannot be connected neither.
\end{proof}

Let now $\w_0, \dots, \w_d$ be a sequence of elements of the hypercube such that the associated graph $H$
is not connected. Thus, there exists a decomposition $W = W_1\sqcup W_2$ and $V=V_1\sqcup V_2$ such that 
there is no edge between $V_i$ and $W_j$ provided that $i \neq j$.

Since $|W_1| + |W_2| > |V_1|+ |V_2|$, we can suppose without loss of generality that $|W_1| > |V_1|$. Let 
$r := |V_1|$, and note that $0<r<d$. 
We have the following proposition,  which clearly implies the vanishing theorem above.  
(Recall that for an element $\w\in \mathbb F_2^d$, the support of $\w$ is
the set of all $j\in [d]$ with $w_j =1$.) 
\begin{prop}\label{prop:vanishing} Let $r$ be a non-negative integer. 
 Let $\w_0, \dots, \w_r$ be elements of $\mathbb F_2^d$ with support in a subset $V_1$ of $V=[d]$ 
 of size $r$. Then for any $\v \in \mathbb F_2^d$, we have 
 \[F_{\w_0}\dots F_{\w_r} F_\v =0. \]
\end{prop}

We need the following useful lemma, proved in~\cite[Proposition 4.29]{Kolb1}.
\begin{lemma} \label{lem:u1}We have 
$F_{\e_1} F_{\e_2}\dots F_{\e_d} F_{\e_1+\dots +\e_d} = (-4)^d C_{\bf 0} C_{\e_1} C_{\e_1+\e_2} \dots C_{\e_1 + \dots +\e_d}$.
\end{lemma}

\begin{proof}[Proof of Proposition~\ref{prop:vanishing}]
Consider the inclusion $\eta: \mathbb F_2^r \rightarrow \mathbb F_2^d$ induced by the subset $V_1$, and let $\eta_*: \widetilde \Ch(\square^r) \rightarrow \widetilde\Ch(\square^d)$ be the induced map on the level of Chow rings given by Proposition~\ref{prop:func2}.

 Since $\Ch(\square^r)$ is one dimensional in degree $r+1$, using Lemma~\ref{lem:u1}, we infer the existence of a rational number $a \in \mathbb Z[\frac12]$ such that we have $\widetilde F_{\w_0} \widetilde F_{\w_1}\dots \widetilde F_{\w_r} = a \widetilde F_{\e_1} \dots \widetilde F_{\e_r} \widetilde F_{\e_1+\dots+\e_r}$, 
 where, by an abuse of the notation, $\e_1,\dots, \e_r$ denote the basis of $\mathbb F_2^r \hookrightarrow \mathbb F_2^d$ corresponding to the elements of $V_1$. It follows that we have 
 \[ F_{\w_0} \dots F_{\w_r} =  a F_{\e_1}\dots F_{\e_r} F_{\e_1+\dots+\e_r}\]
 in the Chow ring $\Ch(\square^d)[\frac 12]$, and it will be enough to prove that for any $\v \in \square^d$, we have
 \[F_{\e_1}\dots F_{\e_r} F_{\e_1+\dots+\e_r}F_\v =0.\]
 
 We prove this by induction on $r$. 
 
 \begin{itemize}
  \item For the base of our induction $r=1$, we need to prove that $F_{\e_1}^2 F_\v =0$.  
 \end{itemize}

  We have $2 F_{\e_1}^2 F_\v = F_{\e_1} (F_{\e_1}+F_0) (F_\v-F_{\v+\e_1}) + F_{e_1} (F_{\e_1}-F_0) (F_\v+F_{\v+\e_1}) =0,$ which proves the claim.
  
  \begin{itemize}
   \item Suppose $r\geq 2$ and assume that the statement holds for  $r-1$, we show it for $r$. 
  \end{itemize}
 Write 
 \begin{align*}
  2F_{\e_1}\dots F_{\e_r} F_{\e_1+\dots+\e_r}F_\v =&  F_{\e_1}\dots F_{\e_{r-1}} F_{\e_r} F_{\e_1+\dots+\e_r} (F_\v - F_{\v+\e_r}) \\
  &+ F_{\e_1}\dots F_{\e_{r-1}} F_{\e_r} F_{\e_1+\dots+\e_r} (F_\v + F_{\v+\e_r}) \\
  =& F_{\e_1}\dots F_{\e_{r-1}} F_{\e_r} \bigl(F_{\e_1+\dots+\e_r} + F_{\e_1+\dots+\e_{r-1}}\bigr) (F_\v - F_{\v+\e_r})\\
  &+ F_{\e_1}\dots F_{\e_{r-1}} F_{\e_r} \bigl(F_{\e_1+\dots+\e_r} - F_{\e_1+\dots+\e_{r-1}}\bigr) (F_\v + F_{\v+\e_r}),
 \end{align*}
which is zero by the relation $(\mathscr R^*3)$ satisfied by the generators $F_\w$ of $\widetilde \Ch(\square^d)$. In the last equalities, we used the induction assumption that $F_{\e_1}\dots F_{\e_{r-1}}F_{\e_1+\dots+\e_{r-1}} F_{\e_r}=0.$
\end{proof}

\begin{remark}\rm
It would be interesting to find a combinatorial formula for the degree with respect to the dual basis 
$F_{\bf w}$. For any collection of $d+1$ elements $(\w_0, \dots, \w_d) \in (\mathbb F_2^d)^{d+1}$, we have 
\[\deg(F_{\w_0} \dots F_{\w_d}) = \sum_{\u_0, \dots, \u_d} (-1)^{\sum_{i} \langle\u_i,\w_i\rangle} \deg(C_{\u_0} 
\dots C_{\u_d}).\]
So the question can be reformulated in asking for the Fourier transform in $(\mathbb F_2^d)^{d+1}$ 
of the degree map on $(\mathbb F_2^d)^{d+1}$ given by Theorem~\ref{thm:main1-intro}. 
\end{remark}

 \end{document}